\documentclass[12pt]{amsart}

\usepackage{amsfonts}
\usepackage{amssymb}
\usepackage{amscd}
\usepackage{graphicx}
\usepackage{hyperref}

\newtheorem{theorem}{Theorem}[section]

\newtheorem{lemma}[theorem]{Lemma}

\newtheorem{corollary}[theorem]{Corollary}
\newtheorem{conjecture}[theorem]{Conjecture}

\theoremstyle{remark} 
\newtheorem{definition}[theorem]{Definition}
\newtheorem{example}[theorem]{Example}
\newtheorem{remark}[theorem]{Remark} 

\newcommand{\kerr}{\mbox{Ker } }
\newcommand{\cokerr}{\mbox{Coker } }

\newcommand{\s}{\mathfrak{s}}

\input xy
\xyoption{all}

\begin{document}

\title{Heegaard Floer groups of Dehn surgeries}
\author{Stanislav Jabuka}

\address{Stanislav Jabuka, Department of Mathematics and Statistics MS 0084, University of Nevada, Reno, NV 89557}

\email{jabuka@unr.edu}

\subjclass{Primary 57M25, 57M27; Secondary 57R58}

\keywords{Heegaard Floer, Dehn surgery, knots, Cabling Conjecture.}

\thanks{This work was partially supported by grant \#246123 from the Simons Foundation.}

\begin{abstract}
We use an algorithm by Ozsv\'ath and Szab\'o to find closed formulae for the ranks of the hat version of the Heegaard Floer homology groups for non-zero Dehn surgeries on knots in $S^3$. As applications we provide new bounds on the number of distinct ranks of the Heegaard Floer groups a Dehn surgery can have. These in turn give a new lower bound on the rational Dehn surgery genus of a rational homology 3-sphere. We also provide novel obstructions for a knot to be a potential counterexample to the Cabling Conjecture. 
\end{abstract}
\maketitle

%%%%%%%%%%
%%%%%%%%%%
%%%%%%%%%%
%%%%%%%%%%
%%%%%%%%%%
\section{Introduction}
%%%
%%%
\subsection{Background and results} In their article \cite{Peter21}, Ozsv\'ath and Szab\'o derive an algorithm for computing the Heegaard Floer groups $\widehat{HF}(S^3_{p/q}(K),\s)$ of the 3-manifold $S^3_{p/q}(K)$ obtained as $p/q$-framed Dehn surgery on the knot $K\subset S^3$, with $\s$ a spin$^c$-structure on $S^3_{p/q}(K)$. The presented algorithm takes as input complexes $\hat A_s, \, \hat B_s$, $s\in \mathbb Z$, derived from the knot chain complex $CFK^\infty (K)$, along with chain maps $\hat v_s, \hat h_s :\hat A_s\to \hat B_s$. These complexes and maps are organized into a \lq\lq master chain complex\rq\rq \, % (a mapping cone complex associated to a chain map between complexes) 
whose homology is the group  $\widehat{HF}(S^3_{p/q}(K),\s)$. 

Ozs\'ath and Szab\'o's  algorithm is not hard to implement for a concrete knot $K\subset S^3$, and readily leads to computations of $\widehat{HF}(S^3_{p/q}(K),\s)$. However, in the interest of broader applications, we focus in this work on evaluating their algorithm in general, without specifying a concrete knot. The result are closed formulae for the ranks of $\widehat{HF}(S^3_{p/q}(K),\s)$, expressed in terms of data coming from the same complexes $\hat A_s, \hat B_s$ and maps $\hat v_s, \, \hat h_s$ alluded to above, but without the need for further computations.  A result akin to the ones presented here is obtained in \cite{Peter21} for the {\em total Heegaard Floer group}
$$\widehat{HF}(S^3_{p/q}(K)) := \oplus _{\s \in Spin^c(S^3_{p/q})} \widehat{HF}(S^3_{p/q}(K),\s).$$
The result is expressed in terms of the knot invariant $\nu:=\nu(K)$ defined as 
\begin{equation} \label{DefinitionOfNu}
\nu(K) = \min\{ s \in \mathbb Z \, |\, \hat v_s:\hat A_s \to \widehat{CF}(S^3) \mbox{ induces a nontrivial map in homology.} \}.
\end{equation}
It is then proved in \cite{Peter21} (Proposition 9.6), under the assumptions that $\gcd (p,q)=1$, $p\ne 0$, $q>0$ and $\nu(K)\ge \nu(-K)$ (where $-K$ is the mirror of the knot $K$, with reversed orientation), that  
\begin{equation} \label{TotalRankOfHF}
\mbox{rk} \, \widehat{HF}(S^3_{p/q}(K)) = \left\{
\begin{array}{cl}
p + 2\max (0, (2\nu -1)q-p) +q \displaystyle \sum _{s\in \mathbb Z} (\mbox{rk } H_*(\hat A_s)-1) & ; \nu >0 \mbox{ or } p>0, \cr
& \cr
|p|+ q \displaystyle  \sum _{s\in \mathbb Z} (\mbox{rk } H_*(\hat A_s)-1) & ; \nu =0 \mbox{ \& } p<0.
\end{array}
\right. 
\end{equation}
To state our results we first introduce some additional notation. For a fixed pair of relatively prime, non-zero integers $p,q\in \mathbb Z$, we define the function $\varphi_{p,q,i} :\mathbb Z\to \mathbb Z$, or $\varphi _i$ for short, as
%%%
%%%
\begin{equation}  \label{CombinatorialInput1}
\varphi _i(s) = \text{Cardinality of the set }\textstyle \left\{ n\in \mathbb Z \, \big|\,  \left\lfloor \frac{i+p\cdot n}{q} \right\rfloor = s \right\}. 
\end{equation}
%%%
%%%
Note that $\varphi _i(s)$ depends on $i$ only through its modulus $[i]$ with respect to $p$. When we wish to emphasize this point we shall write $\varphi _{[i]}(s)$ instead of $\varphi _i(s)$. Note also that for all $s\in \mathbb Z$ the following relation holds
\begin{equation} \label{SummationOverI}
\sum _{[i] \in \mathbb Z/p\mathbb Z} \varphi _{[i]}(s) = |q|.
\end{equation}
It was observed in \cite{Peter21} that $\nu(K) = \tau(K)$ or $\nu (K)=\tau(K)+1$, where $\tau(K)$ is the Ozsv\'ath-Szab\'o concordance invariant from \cite{Peter11}. Since $\tau(-K) = -\tau(K)$ we can always assume that $\nu(K)\ge 0$ by passing to the mirror of $K$ if necessary. This does not pose any restriction to the next theorem, as $S^3_{p/q}(-K) \cong - S^3_{-p/q}(K)$ (as oriented manifolds) and as $\mbox{rk } \widehat{HF}(-Y,\s) = \mbox{rk } \widehat{HF}(Y,\s)$ for any oriented, closed 3-manifold $Y$ and any choice of $\s\in Spin^c(Y)$. Going forward we shall rely on the following notational shortcut:
\begin{equation} \label{Ssubi}
\begin{array}{rl}
\mathcal S_{[i]} & = \displaystyle \sum _{s\in \mathbb Z} \varphi _{[i]}(s) \left( \text{rk } H_*(\hat A_s) - 1\right). % \cr
%& = \displaystyle  \sum _{s= 0}^{g-1} \left( \sum _{|t|\le s}\varphi _{[i]}(t)\right) [\text{rk }  H_*(\hat A_{s}) -\text{rk } H_*(\hat A_{s+1})  ] .
\end{array}
\end{equation}
%
%The equality of the two right-hand sides above is easily demonstrated by relying on the facts that $\text{rk }H_*(\hat A_{-s}) = \text{rk }H_*(\hat A_s)$ for each $s\in \mathbb Z$,  and that $\text{rk } H_*(\hat A_s) = 1$ whenever $|s|\ge g$, where $g$ is the Seifert genus of $K$ \cite{Peter7}. In particular, each $\mathcal S_{[i]}$ is computed by evaluating a finite sum. 
We note that $\text{rk } H_*(\hat A_s) = 1$ whenever $|s|\ge g$, with $g$ the Seifert genus of $K$ \cite{Peter7}. In particular $\mathcal S_{[i]}$ is computed by evaluating a finite sum. 

\begin{theorem} \label{main}
Let $p,q$ be relatively prime, non-zero integers with $q>0$, and let $K$ be a knot in $S^3$ with $\nu = \nu (K) \ge 0$ and with Seifert genus $g$. Then there is an identification of $Spin^c(S^3_{p/q}(K))$ with $\mathbb Z/p\mathbb Z$, such that the ranks of $\widehat{HF}(S^3_{p/q}(K),[i])$ are given as indicated in the two cases below, with $\varphi_{[i]}(s)$ and  $S_{[i]}$ as in \eqref{CombinatorialInput1} and \eqref{Ssubi} respectively. 
%%%
%%% 
\begin{itemize}
\item[(i)] If $\nu>0$ then
%%%%
$$
rk\,  \widehat{HF}(S^3_{p/q}(K),[i]) = 
\left\{
\begin{array}{rl} 
1+ \mathcal S_{[i]} & \quad ; \quad 0<(2\nu-1)q\le p, \cr
%%%
-1 +2 \sum_{|s|<\nu} \varphi_{[i]}(s) +  \mathcal S_{[i]} &  \quad ; \quad 0<p\le (2\nu-1)q, \cr
%%%
1+2 \sum _{|s|<\nu} \varphi _{[i]}(s)+ \mathcal S_{[i]}  &  \quad ; \quad 0>p.
\end{array} 
\right.
$$
\item[(ii)] If $\nu=0$ then %and  $rk(\hat h_0+\hat v_0) =1$, then
%%%%
%%%%
$$
rk\,  \widehat{HF}(S^3_{p/q}(K),[i]) = 
\left\{
\begin{array}{rl}
-1 +  \mathcal S_{[i]}  & \quad ; \quad \begin{array}[t]{l} p< 0, \text{rk }(\hat v_0+\hat h_0)=2, \text{ and }\cr
\lfloor \frac{i+ps}{q}\rfloor = 0 \text{ for some } s\in \mathbb Z, \end{array}\cr 
1 +  \mathcal S_{[i]}  & \quad ; \quad\text{otherwise. }
\end{array}
\right.
$$
%%%
%%%
\end{itemize}
\end{theorem}
It is observed in \cite{Peter21} that the case of a knot $K$ with $\nu(K)=0$ and $\text{rk }(\hat v_{0,*} + \hat h_{0,*})=2$ can be avoided by passing to its mirror knot $-K$. In situations where $-p/q$-surgery on $-K$ carries the same information one wishes to read off from the $p/q$-surgery on $K$, the first case in (ii) in Theorem \ref{main} is superfluous.
%%%
%%%

Equation \eqref{TotalRankOfHF} is found to be an easy consequence of Theorem \ref{main} and equation \eqref{SummationOverI}. For example, if $-p, \nu>0$ then 
\begin{align*} 
\sum _{[i]\in \mathbb Z/p\mathbb Z} \left(  1+2\sum _{|s|<\nu} \varphi_{[i]}(s)+ \mathcal S_{[i]} \right) & =  -p + 2 \sum _{|s|<\nu}\sum _{[i]\in \mathbb Z/p\mathbb Z} \varphi_{[i]}(s) + \sum _{s\in \mathbb Z} \sum _{[i]\in \mathbb Z/p\mathbb Z} \varphi_{[i]}(s) (\mbox{rk } H_*(\hat A_s)-1), \cr
%%%
& = -p + 2 \sum _{|s|<\nu} q + q \sum _{s\in \mathbb Z}(\mbox{rk } H_*(\hat A_s)-1), \cr
%%%
& = -p + 2q(2\nu-1) + q \sum _{s\in \mathbb Z}(\mbox{rk } H_*(\hat A_s)-1), \cr
%%%
& = p+2\max(0,(2\nu-1)q-p)+ q \sum _{s\in \mathbb Z}(\mbox{rk } H_*(\hat A_s)-1).
\end{align*}
%%%
%%%
The results of Theorem \ref{main} become particularly simple for integral sugeries, for which 
$$\varphi_{p,1,[i]}(s) = \left\{
\begin{array}{cl}
1 & \quad ; \quad s\equiv i \,(\text{mod } p), \cr
0 & \quad ; \quad s\not\equiv i \,(\text{mod } p).
\end{array}
\right.
$$
\begin{corollary} \label{CorollaryAboutIntegerSurgeries}
Assume the hypotheses of Theorem \ref{main} and pick $q=1$. Let $\mathcal I_{[i]}$ be the (possibly empty) set of integers $\mathcal I_{[i]} = \{ s\in \mathbb Z\,|\,  |s|<g \text{ and } s\equiv i \,(\text{mod } p)\}$, and write $2\nu-1 = pn+r$ with $n$ an integer and with $r\in \{0,\dots,|p|-1\}$. If $\nu>0$ define the set of indices $\mathcal J\subset \mathbb Z/p\mathbb Z$ as $\mathcal J = \{[-\nu+j]\, |\, j=1,\dots, r\}$ (if $r=0$ then $\mathcal J=\emptyset$). Then the ranks of $\widehat{HF}(S^3_p(K),[i])$ are given by 
\begin{itemize}
\item[(i)] If $\nu >0$ then 
$$
rk\,  \widehat{HF}(S^3_{p/q}(K),[i]) = 
\left\{
\begin{array}{rl} 
1+ \sum_{s\in \mathcal I_{[i]}} (\text{rk }H_*(\hat A_s) -1) & \quad ; \quad 0<(2\nu-1)q\le p, \cr &\cr
%%%
2n+1+ \sum_{s\in \mathcal I_{[i]}} (\text{rk }H_*(\hat A_s) -1) &  \quad ; \quad 0<p\le (2\nu-1)q, [i]\in \mathcal J\cr
2n-1+ \sum_{s\in \mathcal I_{[i]}} (\text{rk }H_*(\hat A_s) -1) &  \quad ; \quad 0<p\le (2\nu-1)q, [i]\notin \mathcal J\cr & \cr
%%%
2n+3+ \sum_{s\in \mathcal I_{[i]}} (\text{rk }H_*(\hat A_s) -1)  &  \quad ; \quad 0>p, [i]\in \mathcal J, \cr
2n+1+ \sum_{s\in \mathcal I_{[i]}} (\text{rk }H_*(\hat A_s) -1)  &  \quad ; \quad 0>p, [i]\notin \mathcal J.
\end{array} 
\right.
$$
\item[(ii)] If $\nu=0$ then %and  $rk(\hat h_0+\hat v_0) =1$, then
%%%%
%%%%
$$
rk\,  \widehat{HF}(S^3_{p/q}(K),[i]) = 
\left\{
\begin{array}{rl}
-1 + \sum_{s\in \mathcal I_{[i]}} (\text{rk }H_*(\hat A_s) -1) & \quad ; \quad \begin{array}[t]{l} p< 0, \text{rk }(\hat v_0+\hat h_0)=2, \text{ and }\cr
\lfloor \frac{i+ps}{q}\rfloor = 0 \text{ for some } s\in \mathbb Z, \end{array}\cr 
1 + \sum_{s\in \mathcal I_{[i]}} (\text{rk }H_*(\hat A_s) -1)  & \quad ; \quad\text{otherwise. }
\end{array}
\right.
$$
\end{itemize}
\end{corollary}
%%%%%%%%%%%%%%%%%%%%%%%%%%%%%%%%%%%%%%%%%%%%%%%%%%%%%%%%%%%%%%%%%%%%%%%%%%%%%%%%%%%%%%%%%%%%%%%%%%%%%%%%%%%%%%%%%%%%%%%%%%%%%%%%%%%%%%%%%%%%%%%%%%%%%%%%%%%%%%%%%%%%%%%%%%%%%%%%%%%%%%%%%%%%%%%%%%%%%%%%%%%%%%%%%%%%%%%%%%%%
Next we turn to applications of Theorem \ref{main} and Corollary \ref{CorollaryAboutIntegerSurgeries}. The first application concerns the Dehn surgery genera of rational homology 3-spheres introduced in \cite{jabuka1}, the second gives new obstructions for a knot to possibly provide a counterexample to the Cabling Conjecture \cite{AcunaGonzalesShort}. 
%%%
%%%
\subsection{Dehn surgery genera of rational homology 3-spheres} 
%%%
%%%
In \cite{jabuka1} we defined the {\em integral and rational Dehn surgery genera $g_\mathbb Z(Y)$ and $g_\mathbb Q(Y)$} of a rational homology $3$-sphere $Y$ as 
$$\begin{array}{l}
g_\mathbb Z(Y) = \left\{
\begin{array}{cl}
\min\{ g(K) \, |\, Y=S^3_r(K), r\in \mathbb Z\} & ; \quad \text{If $Y=S^3_r(K)$ for some $K$,} \cr  
\infty & ; \quad \text{Otherwise.}
\end{array}
\right. \cr \cr
%%%
%%%
g_\mathbb Q(Y) = \left\{
\begin{array}{cl}
\min\{ g(K) \, |\, Y=S^3_r(K), r\in \mathbb Q\} & ; \quad \text{If $Y=S^3_r(K)$ for some $K$,} \cr  
\infty & ; \quad \text{Otherwise.}
\end{array}
\right.
\end{array}
$$
Both types of genera should be viewed as measures of complexity for rational homology 3-spheres.  Note that $g_\mathbb Z(Y) \ge g_ \mathbb Q(Y)$ and examples are provided in \cite{jabuka1} showing that $g_\mathbb Z(Y) - g_\mathbb Q(Y)$ can become arbitrarily large (while staying finite). One of the results from \cite{jabuka1} is the lower bound 
\begin{equation} \label{IntegralDehnSurgeryGenusBound}
2g_\mathbb Z(Y) - 1\ge |H_1(Y;\mathbb Z)| - \ell,  
\end{equation}
where $\ell$ is the number of $L$-structures on $Y$ (an $L$-structure on $Y$ is a spin$^c$-structure $\s$ on $Y$ with $\widehat{HF}(Y,\s) \cong \mathbb Z$). This bound is sharp for many examples and can be used to determine the integral Dehn surgery genus for some 3-manifolds. Techniques from \cite{jabuka1} do not suffice to provide an analogous lower bound on $g_\mathbb Q(Y)$. This missing bound can now be established courtesy of Theorem \ref{main}. 
%%%
%%%
\begin{theorem} \label{Estimate}
Let $Y$ be the result of $p/q$-framed surgery on a genus $g$ knot $K$ in $S^3$. Then the cardinality of the set 
\begin{equation}\label{DefinitionOfTheSetMathcalR}
\mathcal {R} = \{ \mbox{rk } \widehat{HF}(Y,\s)\, |\, \s\in Spin^c(Y)\},
\end{equation}
does not exceed $g+1$. 
\end{theorem}
\begin{remark} \label{RemarkAboutBoundOfSpinCStructures}
The cardinality of $\mathcal R$ is also bounded by $\frac{|p|+1}{2}$ if $p$ is odd, and by $\frac{|p|+2}{2}$ if $p$ is even, placing the utility of Theorem \ref{Estimate} into the range of $|p|\ge 2g+1$.
\end{remark}
%%%
%%%   

%%%
%%%
\begin{corollary}
The rational Dehn surgery genus $g_\mathbb Q(Y)$ of a rational homology $3$-sphere is bounded from below as 
$$ g_\mathbb Q(Y)\ge |\mathcal R|-1, $$
with $\mathcal R$ as in \eqref{DefinitionOfTheSetMathcalR}.
\end{corollary}

\begin{example} Consider the genus 2 knot $K=6_2$ from the knot tables. For this knot $\nu(K) = 1$ and 
$$ \mbox{rk }H_*(\hat A_s) = \left\{
\begin{array}{cl}
1 & \quad ; \quad |s| \ne 1, \cr
3 & \quad ;\quad |s| = 1.
\end{array}
\right.
$$
Theorem \ref{main} implies that 
$$\mbox{rk } \widehat{HF}(S^3_{5/2}(K), [i]) = \left\{ 
\begin{array}{cl}
1 & \quad ; \quad [i] = 0, 1, \cr
3 & \quad ; \quad [i] = 2, 4, \cr
5 & \quad ; \quad [i] = 3. 
\end{array}
\right.$$
We conclude that $S^3_{5/2}(6_2)$ cannot be obtained by surgery on a knot of genus 1, and therefore $g_\mathbb Q(S^3_{5/2}(6_2)) = 2$. 
\end{example}
%%%
%%%
\begin{example}
For the genus 3 knot $K=-8_5$ with $\nu=2$, $\frac{7}{4}$--surgery yields a 3-manifold whose Heegaard Floer groups have ranks $(7, 9, 9, 7, 3, 1, 3)$ for the spin$^c$-structures $[i]=0,\dots, 6$ respectively. It follows that $g_\mathbb Q(S^3_{7/4}(-8_5))=3$, but note that \eqref{IntegralDehnSurgeryGenusBound} shows that $g_\mathbb Z(S^3_{7/4}(-8_5))\ge 4$. 
%
%On the other hand, $9$--surgery on the same knot yields Heegaard Floer groups with ranks $(5, 1, 3, 1, 1, 1, 1, 3, 1)$, which only leads to $g_\mathbb Q(S^3_{9}(-8_5))\ge 2$.
\end{example}
%%%%%%%%%%%%%%%%%%%%%%%%%%%%%%%%%%%%%%%%%%%%%%%%%%%%%%%%%%%%%%%%%%%%%%%%%%%%%%%%%%%%%%%%%%%%%%%%%%%%%%%%%%%%%%%%%%%%%%%%%%%%%%%%%%%%%%%%%%%%%%%%%%%%%%%%%%%%%%%%%%%%%%%%%%%%%%%%%%%%%%%%%%%%%%%%%%%%%%%%%%%%%%%%%%%%%%%%%%%%
\subsection{Obstructing reducible surgeries} \label{SectionOnCablingConjecture}
%%%
%%%

In this section we discuss how Theorem \ref{main} can be used to obstruct a Dehn surgery from being a reducible 3-manifold. A lot is known about this question, and we start by reminding the reader of the relevant results. This application was suggested to us by Tye Lidman, whose input we gratefully acknowledge. 

Let $r$ be a rational number and $K$ a nontrivial knot in $S^3$. We say that $S^3_r(K)$ is {\em reducible} if it possesses an essential 2-sphere, that is a 2-sphere that doesn't bound a 3-ball, in which case $r$ is called a {\em reducing slope} for $K$. Work of Gabai  \cite{Gabai} shows that if $S_r^3(K)$ is reducible then $r\ne 0$ and $S_r^3(K)$ is a connected sum. Gordon and Luecke in \cite{GordonLuecke1} show that a reducing slope is necessarily an integer, and they show in \cite{GordonLuecke3} that the geometric intersection number between any two reducing slopes is 1. This latter fact limits the number of possible reducing slopes of $K$ to two, and if there are two, they are consecutive integers. Further work of Gordon and Luecke \cite{GordonLuecke2} shows that if $S_r^3(K)$ is reducible, one of its connected summands is a lens space. Moreover, an integral homology sphere obtained by Dehn surgery is irreducible, so that a reducing slope $r$ satisfies $r\ne \pm 1$. Work of Matignon and Sayari \cite{MatignonSayari} shows that an integral reducing slope $r$ for a genus $g$ knot $K$ which isn't a cable knot, satisfies the bound $|r|\le 2g-1$. For other recent results please see \cite{Baker, Greene, HomKarakurtLidman}. 

Cable knots are known to have reducible surgeries, indeed the $pq$-surgery of a $(p,q)$-cable of the knot $K$ yields the reducible 3-manifold $S^3_{p/q}(K)\# L(q,p)$. The standing conjecture is that cable knots are the only knots that admit reducible surgeries.
\vskip1mm
%%%
%%%
%%%
\begin{conjecture}[{\em Cabling Conjecture}, Gonzalez-Acu\~na -- Short, \cite{AcunaGonzalesShort}] 
If $K$ is a non-trivial knot in $S^3$ which admits a reducible Dehn surgery, then $K$ is a cable knot. 
\end{conjecture}
%%%
%%%
%%%
The Cabling Conjecture is known to hold for several families of knots, including satellite knots \cite{Scharlemann}, alternating knots \cite{MenascoThistlethwaite} and genus one knots \cite{BoyerZhang}. 

In summary, if $r$-framed surgery on a knot $K$ of genus $g$ is reducible, then $K$ is hyperbolic and non-alternating with $g\ge 2$, $r$ is an integer in the range $1<|r|\le 2g-1$, and $S_r^3(K)\cong Y\# L(a,b)$ for some 3-manifold $Y$ and with $a$ dividing $r$. 

The Heegaard Floer homology groups are well suited to provide an obstruction for the reducibility of a rational homology sphere, as the ranks of said groups are easily computed for a connected sum, cf. \cite{Peter2}:
$$ \text{rk }\widehat{HF}(Y_1\#Y_2, \s_1\#s_2) = \left(\text{rk } \widehat{HF}(Y_1,\s_1) \right) \cdot \left(\text{rk } \widehat{HF}(Y_2,\s_2) \right).$$
Additionally, all lens spaces are $L$-spaces, that is 3-manifolds $Y$ with $\widehat{HF}(Y,\s) \cong \mathbb Z$ for all $\s \in Spin^c(Y)$. These two observation show that for a reducible surgery, the ranks of the Heegaard Floer groups have to appear with prescribed multiplicities. 
%%%
%%%
\begin{theorem} \label{PertainingToCablingConjecture}
Let $K$ be a nontrivial knot of genus $g$ and suppose that $p$ surgery on $K$ yields a reducible manifold, with $p$ an integer in the range $1<|p|\le 2g-1$. Then there exist integers $a,b$ with $a>1$ and $p = a\cdot b$, such that the list of integers $\mathcal L$ defined below, has all of its entries occur with a multiplicity that is divisible by $a$. 

The list $\mathcal L$ is defined as follows. If $\nu>0$ and $2\nu-1\le p$, or if $\nu =0$, let $\mathcal L$ be 
$$\mathcal L = \left\{\textstyle \sum_{s\in \mathcal I_{[i]}} \text{rk } (H_*(\hat A_s) - 1) \, \big| \,  i=0,\dots,|p|-1.\right\}. $$
If $\nu>0$ and $p\le 2\nu -1$, set $\mathcal L = \mathcal A \cup \mathcal B$ with 
$$
\begin{array}{l}\mathcal A = \left\{2+ \sum_{s\in \mathcal I_{[i]}} (\text{rk } H_*(\hat A_s) - 1)\, \big| \,  i=-\nu+1,\dots,-\nu+r.\right\}, \text{ and } \cr \cr
\mathcal B = \left\{\sum_{s\in \mathcal I_{[i]}} \text{rk } (H_*(\hat A_s) - 1) \, \big| \,  k=-\nu+r+1,\dots,-\nu+|p|.\right\}.
\end{array}
$$
In the above, $r\in \{0,\dots,|p|-1\}$ is uniquely determined by writing $2\nu-1 = np+r$ with $n$ an integer. The set $\mathcal I_{[i]}$ is given by $\mathcal I_{[i]} = \{s\in \mathbb Z \, |\, |s|<g \text{ and } s\equiv i \,(\text{mod } p)\}$. If $r=0$ then $\mathcal A$ is the empty set, and the second definition of $\mathcal L$ reduces to the first. 
\end{theorem}
%%%
%%%
\begin{example}
Consider the hyperbolic, non-alternating genus 2 knot $K=8_{20}$. For this knot one calculates 
$$\nu=0, \quad \quad \text{rk } H_*(\hat A_{\pm 1}) = 3, \quad \text{rk } H_*(\hat A_{s}) = 1\, \text{ for all } s\ne \pm 1. $$
Given these, the lists $\mathcal L$ from the preceding theorem for values of $p$ with $2\le|p|\le 3$, are easily computed, and are given by  
\vskip2mm
%%%
%%%
\begin{center}
\begin{tabular}{c|c}
$p$ & $\mathcal L$ \cr \hline \hline
$\pm 2$ & $\{4, 0\}$ \cr
$\pm 3$ & $\{2,0,2\}$ 
\end{tabular}

\end{center}
%%%
%%%
\vskip2mm
It is readily verified that none of the values of $p$ listed in the tables, possess a factor $a>1$ so that the associated $\mathcal L$ has its elements appear with multiplicity $a$. Accordingly, no surgery on $K=8_{20}$ can yield a counterexample to the Cabling Conjecture. 
\end{example}
%%%
%%%
\begin{example} Consider here the hyperbolic, non-alternating knot $K=-9_{47}$ of genus 3 (recall that $-K$ denotes the mirror of $K$). For this knot one obtains
$$\nu=1, \quad \quad \text{rk } H_*(\hat A_{\pm 2}) = 3, \quad H_*(\hat A_{\pm 1}) = 5, \quad  \text{rk } H_*(\hat A_{s}) = 1\, \text{ for all } s\ne \pm 1, \pm 2. $$
The associated lists $\mathcal L$ from Theorem \ref{PertainingToCablingConjecture} for $p$ in the range $2\le |p|\le 5$ are 
\vskip2mm
%%%
%%%
\begin{center}
\begin{tabular}{c|c}
$p$ & $\mathcal L$ \cr \hline \hline
2 & $\{4, 8\}$ \cr
3 & $\{6, 6, 0\}$ \cr
4 & $\{4, 4, 0 , 4\}$ \cr
5 & $\{2, 4, 0, 4, 2\}$ 
\end{tabular}
\quad \quad \quad \quad \quad \quad 
\begin{tabular}{c|c}
$p$ & $\mathcal L$ \cr \hline \hline
-2 & $\{6, 8\}$ \cr
-3 & $\{2, 6, 6\}$ \cr
-4 & $\{2, 4, 4, 4\}$ \cr
-5 & $\{2, 4, 2, 4, 2\}$ 
\end{tabular}
\end{center}
%%%
%%%
\vskip2mm
Theorem \ref{PertainingToCablingConjecture} excludes $-9_{47}$ (and hence $9_{47}$) from providing a counterexample to the Cabling Conjecture. 
\end{example}
%%%%%%%%%%%%%%%%%%%%%%%%%%%%%%%%%%%%%%%%%%%%%%%%%%%%%%%%%%%%%%%%%%%%%%%%%%%%%%%%%%%%%%%%%%%%%%%%%%%%%%%%%%%%%%%%%%%%%%%%%%%%%%%%%%%%%%%%%%%%%%%%%%%%%%%%%%%%%%%%%%%%%%%%%%%%%%%%%%%%
%\subsection{Organization} 
%In Section \ref{RationalSurgeryFormula} we review knot Floer homology and give more details on the rational surgery formula of Ozsv\'ath and Szab\'o from \cite{Peter21}. Section \ref{SectionOnProofs} gives proofs of the results from the introduction. %The final section gives a short table of invariants for knots up to 7 crossings which includes $\nu(K)$ and $\text{rk } H_*(\hat A_s)$, $s\in \mathbb Z$. 

{\bf Acknowledgements } I have enjoyed and benefitted from conversations with Tom Mark and Tye Lidman. 
%%%%%%%%%%%%%%%%%%%%%%%%%%%%%%%%%%%%%%%%%%%%%%%%%%%%%%%%%%%%%%%%%%%%%%%%%%%%%%%%%%%%%%%%%%%%%%%%%%%%%%%%%%%%%%%%%%%%%%%%%%%%%%%%%%%%%%%%%%%%%%%%%%%%%%%%%%%%%%%%%%%%%%%%%%%%%%%%%%%%%%%%%%%%%%%%%%%%%%%%%%%%%%%%%%%%%%%%%%%%%%%%%%%%%%%%%%%%%%%%%%%%%%%%%%%%%%%%%%%%%%%%%%%%%%%%%%%%%%%%%%%%%%%%%%%%%%%%%%%%%%%%%%%%%%%%%%%%%%%%%%%%%%%%%%%%%%%%%%%%%%%%%%%%%%%%
\section{Review of the rational surgery formula} \label{RationalSurgeryFormula}
%%%%
%%%%
%%%%
This section reviews background material on knot Floer homology and explains the rational surgery formula from \cite{Peter21}. Our notation and outline follow those given in \cite{Peter7, Peter21}. 
%%%%%%%%%%%%%%%%%%%%%%%%%%%%%%%%%%%%%%%%%%%%%%%%%%%%%%%%%%%%%%%%%%%%%%%%%%%%%%%%%%%%%%%%%%%%%%%%%%%%%%%%%%%%%%%%%%%%%%%%%%%%%%%%%%%%%%%%%%%%%%%%%%%%%%%%%%%%%%%%%%%%%%%%%%%%%%%%%%%%%%%%%%%%%%%%%%%%%%%%%%%%%%%%%%%%
\subsection{Knot Floer homology}Given a doubly pointed Heegaard diagram for a  knot $K$ in the 3-sphere, Ozsv\'ath and Szab\'o in \cite{Peter7} associate to it a $\mathbb Z^2$-filtered {\em knot Floer chain complex } $\left( CFK^\infty (K), \partial ^\infty \right)$. Its generators over $\mathbb Z$ are of the form $[x,i,j]$ where $x$ ranges through a finite set $\mathcal X$ determined by the doubly pointed Heegaard diagram, and $i,j$ are integers. The knot Floer chain complex has the structure of a $\mathbb Z[U]$-module with the action of $U$ on a generator $[x,i,j]$ given by $U\cdot [x,i,j] = [x,i-1,j-1]$. 

The aforementioned $\mathbb Z^2$-filtration $\mathcal F: \mathcal X \times \mathbb Z^2 \to \mathbb Z^2$ is given by $\mathcal F([x,i,j]) = (i,j)$, and we let $C\{i,j\}$ denote the subgroup of $CFK^\infty(K)$ generated by elements in filtration level $(i,j)$. The differential $\partial ^\infty$ respects the filtration in that 
\begin{equation} \label{RespectfulDifferential}
\partial ^\infty \left(C\{i_0,j_0\}  \right) \subseteq \bigoplus _{i\le i_0,\,  j\le j_0} C\{i,j\}.
\end{equation}
%
%It shall be useful to picture the groups $C\{i,j\}$ as sitting at coordinate $(i,j)$ in the coordinate plane.

Property \eqref{RespectfulDifferential} of the differential $\partial ^\infty$ allows for the definition of a number of subchain complexes of $CFK^\infty(K)$ and associated quotient chain complexes. They are parametrized by $s, t\in \mathbb Z$:
\begin{equation} \label{TheChainComplexes}
\begin{array}{ll}
 C\{i\le s \text{ and } j\le t\}  :=  \bigoplus _{i\le s,\,  j\le t} C\{i,j\}, \quad \quad \quad & 
\hat A_s  : = \displaystyle \frac{C\{i\le 0 \text{ and } j\le s\}}{C\{i\le -1 \text{ and } j\le s-1\}},
 \cr & \cr
C\{i\le s\}  : =  \bigoplus _{i\le s} C\{i,j\}, \quad & 
\hat B_s:= \displaystyle  \frac{C\{i\le 0\}}{C\{i\le -1\}} =: C\{i=0\}, \cr & \cr
C\{j\le t\}  : =  \bigoplus _{j\le t} C\{i,j\}, \quad & 
\hat C_t := \displaystyle  \frac{C\{j\le 0\}}{C\{j\le -1\}}=: C\{j=0\}.
\end{array} 
\end{equation} 
The $\mathbb Z[U]$-module structure extends to these complexes, though it is trivial on $\hat A_s$, $\hat B_s$ and $\hat C_s$. We note that $\hat B_s$ is independent of $s$. Indeed the role of the index $s$ is merely for bookkeeping and will prove useful going forward. When we don't care about the label, we shall simply write $\hat B$ to mean $\hat B_s$. Analogous comments apply to $\hat C_t$.  %It is useful to visualize $\hat A_s$ as \dots

As explained in \cite{Peter7}, there is a $\mathbb Z[U]$-module quasi-isomorphism of chain complexes
$J: C\{j\le 0 \} \to C\{i\le 0 \},$ induced by interchanging the two base points in the doubly pointed diagram used to define $CFK^\infty(K)$. This induces a quasi-isomorphism $\hat J : C\{j=0\} \to C\{i=0\}$. For any $s\in \mathbb Z$ let $\pi_s:C\{i\le 0 \text{ and } j\le s\}\to C\{i\le 0\}$ be the projection map, and let $\hat \pi_s:\hat A_s \to \hat B$ be the associated map of quotient complexes. Similarly, let $\hat \Pi_s : U^s(\hat A_s) \to \hat C$ be induced by the projection map $\Pi_s : C\{i\le -s \text{ and } j\le 0\}\to C\{j\le 0\}$. These maps give rise to two chain maps $\hat v_s, \hat h_s : \hat A_s \to \hat B$ defined as   
\begin{equation} \label{VsAndHs}
\hat v_s = \hat \pi_s \quad \quad \text{ and } \quad \quad \hat h_s =  \hat J\circ \hat \Pi_s   \circ U^s .
\end{equation}
It is easy to verify that $\hat v_s$ and $\hat h_{-s}$ are isomorphisms for $s\ge g$, while $\hat v_s=0=\hat h_{-s}$ whenever $s<-g$. 

Equation \eqref{DefinitionOfNu} from the introduction defines the knot invariant $\nu = \nu(K)$ as 
$$
\nu(K) = \min\{ s \in \mathbb Z \, |\, (\hat v_s)_*:H_*(\hat A_s) \to \mathbb Z \mbox{ is nontrivial.} \}.
$$
\begin{remark} \label{RemarkAboutNuAndVsAndHs}
The definition of $\nu(K)$ implies that $(\hat v_s)_*=0$ for $s< \nu$, and by symmetry that $(\hat h_{s})_*=0$ for $s>-\nu$. It is not hard to see that conversely, $(\hat v_s)_*$ and $(\hat h_s)_*$ are nontrivial maps for $s\ge \nu$ and $s\le -\nu$ respectively. 
\end{remark}
%%%%%%%%%%%%%%%%%%%%%%%%%%%%%%%%%%%%%%%%%%%%%%%%%%%%%%%%%%%%%%%%%%%%%%%%%%%%%%%%%%%%%%%%%%%%%%%%%%%%%%%%%%%%%%%%%%%%%%%%%%%%%%%%%%%%%%%%%%%%%%%%%%%%%%%%%%%%%%%%%%%%%%%%%%%%%%%%%%%%%%%%%%%%%%%%%%%%%%%%%%%%%%%%%%%%
\subsection{The rational surgery formula}
In this section we fix two nonzero relatively prime integers $p,q$. For $i\in \mathbb Z$ define chain complexes $\hat{\mathbb  A}_i$ and $\hat{\mathbb B}_i$ via 
\begin{equation} \label{AiAndBiComplexes}
\hat{\mathbb A}_i = \displaystyle \bigoplus _{s\in \mathbb Z} \left( s,\hat A_{\lfloor \frac{i+ps}{q} \rfloor} \right)\quad \text{ and } \quad \hat{\mathbb B}_i =  \displaystyle \bigoplus _{s\in \mathbb Z} (s, \hat B) .
\end{equation}
Here $(s,\hat A_{\lfloor\frac{i+ps}{q} \rfloor})$ and $(s,\hat B)$ are copies of $A_{\lfloor\frac{i+ps}{q} \rfloor}$ and $\hat B$ respectively.  The complex $\hat{\mathbb B}_i$ is independent of $i$, while $\hat{\mathbb A}_i$ depends on $i$ only through its modulus $[i]$ with respect to $p$. To emphasize this we shall write $\hat{\mathbb A}_{[i]}$ and $\hat{\mathbb B}_{[i]}$.

We define chain maps $\hat v, \hat h :\hat{\mathbb A}_{[i]} \to \hat{\mathbb B}_{[i]}$ by requiring that $\hat v$ map the summand $(s, \hat A_{\lfloor \frac{i+ps}{q}\rfloor})$ of $\hat{\mathbb A}_{[i]}$ to the summand $(s,\hat B)$ of $\hat{\mathbb B}_{i}$ via $\hat v_{\lfloor \frac{i+ps}{q}\rfloor}$. Similarly, we require that $\hat h$ map $(s, \hat A_{\lfloor \frac{i+ps}{q}\rfloor})$ to $(s-1,\hat B)$ via $\hat h_{\lfloor \frac{i+ps}{q}\rfloor}$. 

Let $r=p/q$ and define a new chain map $\hat{\mathbb D}_{r,[i]}:\hat{\mathbb A}_{[i]} \to \hat{\mathbb B}_{i]}$ by settting 
$$ \hat{\mathbb D}_{r,[i]} \left( \{(s,a_s)\}_{s\in \mathbb{Z}} \right) = \{(s,b_s)\}_{s\in \mathbb{Z}} \quad \quad 
\mbox{ with } \quad \quad 
b_s = \hat v_{\lfloor \frac{i+ps}{q} \rfloor}(a_s) + \hat h_{\lfloor \frac{i+p(s-1)}{q} \rfloor} (a_{s-1}). $$
Let $\hat{\mathbb X}_{r,[i]}$ be the mapping cone of $\hat{\mathbb D}_{r,[i]}$. We thus arrive at the following algorithm for computing the Heegaard Floer groups of a rational surgery on a knot in $S^3$.
%%%
%%%
%%%
\begin{theorem}[Ozsv\'ath-Szab\'o \cite{Peter21}] \label{RationalSurgeryTheorem}
Let $K\subset S^3$ be a knot and let $p,q\in\mathbb{Z}$ be a pair of relatively prime, nonzero integers. Then there is an affine identification $[i]\mapsto \s_{[i]}$ of $\mathbb Z/p\mathbb Z$ with $Spin^c(Y)$, such that  
$$ \widehat{HF}(S^3_{p/q}(K),\s_{[i]}) \cong H_*(\hat{\mathbb{X}}_{[i],r}).$$
\end{theorem}
%%%
%%%

We finish this section with a discussion regarding computation of the homology groups of the mapping cone $\hat{\mathbb X}_{r,[i]}$. Let $\hat{\mathbb A}_{[i],j}$, $\hat{\mathbb B}_{[i],j}$ and $\hat{\mathbb X}_{r,[i],j}$ denote the degree $j$ subgroups of $\hat{\mathbb A}_{[i]}$, $\hat{\mathbb B}_{[i]}$ and $\hat{\mathbb X}_{r,[i]}$ respectively, and note that $\hat{\mathbb X}_{r,[i],j} = \hat{\mathbb A}_{[i],j} \oplus \hat{\mathbb B}_{[i],j+1}$. Let $\partial _{\hat{\mathbb X}}$ be the differential of $\hat{\mathbb X}_{r,[i]}$, and let  $\partial _{\hat{\mathbb X},j}$ be its restriction to $\hat{\mathbb X}_{r,[i],j}$. Then 
$$ \partial _{\hat{\mathbb X},j} (a,b) = (\partial _{\hat{\mathbb A},j} a , \, \partial _{\hat{\mathbb B},j+1} b+ (-1)^j \,  \hat{\mathbb D}_{r,[i],j} a), $$
where $\partial _{\hat{\mathbb A}}$ and $\partial_{\hat{\mathbb B}}$ are the differentials of $\hat{\mathbb A}_{[i]}$ and $\hat{\mathbb B}_{[i]}$ respectively, while $\partial _{\hat{\mathbb A},j}$ and $\partial_{\hat{\mathbb B},j}$ are their restrictions to $\hat{\mathbb A}_{[i],j}$ and $\hat{\mathbb B}_{[i],j}$. Also, $\hat{\mathbb D}_{r,[i],j}$ is the restriction of $\hat{\mathbb D}_{r,[i]}$ to $\hat{\mathbb A}_{[i],j}$. 

With these details made explicit, it is now easy to verify that 
$$ 0 \to \hat{\mathbb B}_{[i],j+1} \stackrel{\iota}{\longrightarrow} \hat{\mathbb X}_{r,[i],j} \stackrel{\pi}{\longrightarrow} \hat{\mathbb A}_{[i],j}\to 0,  $$
is a short exact sequence of chain complexes, with $\iota$ and $\pi$ being the inclusion and projection maps. The induced long exact sequence (where $\hat{\mathbb D}_{r,[i],j}$ is abbreviated to $\hat{\mathbb D}_j$)
\begin{equation} \label{MappingConeLES} 
\dots \to H_*(\hat{\mathbb A}_{[i],j+1}) \stackrel{(\hat{\mathbb D}_{j+1})_*}{\longrightarrow} H_*(\hat{\mathbb B}_{[i],j+1}) \stackrel{\iota_*}{\longrightarrow} H_*(\hat{\mathbb X}_{r,[i],j}) \stackrel{\pi_*}{\longrightarrow} H_*(\hat{\mathbb A}_{[i],j})\stackrel{(\hat{\mathbb D}_{j})_*}{\longrightarrow} H_*(\hat{\mathbb B}_{[i],j}) \to \dots 
\end{equation}
is easily seen to have its connecting homomorphism equal to $(\hat{\mathbb D}_j)_*$. This latter sequence leads to the next straightforward but useful observation. 
%%%
%%%
\begin{theorem} \label{RankOfMappingCone}
The rank of the homology group $H_*(\hat{\mathbb X}_{r,[i]})$ of the mapping cone $\hat{\mathbb X}_{r,[i]}$ of $\hat{\mathbb D}_{r,[i]}:\hat{\mathbb A}_{[i]}\to \hat{\mathbb B}_{[i]}$, can be computed with the aid of \eqref{MappingConeLES} as
$$\text{rk } H_*(\hat{\mathbb X}_{r,[i]}) = \text{rk }[\kerr (\hat{\mathbb D}_{r,[i]})] + \text{rk } [\cokerr (\hat{\mathbb D}_{r,[i])}]. $$
\end{theorem}
%%%%%%%%%%%%%%%%%%%%%%%%%%%%%%%%%%%%%%%%%%%%%%%%%%%%%%%%%%%%%%%%%%%%%%%%%%%%%%%%%%%%%%%%%%%%%%%%%%%%%%%%%%%%%%%%%%%%%%%%%%%%%%%%%%%%%%%%%%%%%%%%%%%%%%%%%%%%%%%%%%%%%%%%%%%%%%%%%%%%%%%%%%%%%%%%%%%%%%%%%%%%%%%%%%%%%%%%%%%%%%%%%%%%%%%%%%%%%%%%%%%%%%%%%%%%%%%%%%%%%%%%%
%%%%%%%%%%%%%%%%%%%%%%%%%%%%%%%%%%%%%%%%%%%%%%%%%%%%%%%%%%%%%%%%%%%%%%%%%%%%%%%%%%%%%%%%%%%%%%%%%%%%%%%%%%%%%%%%%%%%%%%%%%%%%%%%%%%%%%%%%%%%%%%%%%%%%%%%%%%%%%%%%%%%%%%%%%%%%%%%%%%%%%%%%%%%%%%%%%%%%%%%%%%%%%%%%%%%%%%%%%%%%%%%%%%%%%%%%%%%%%%%%%%%%%%%%%%%%%%%%%%%%%%%%%%%%%%%%%%%%%%%%%%%%%%%%%%%%%%%%%%%%%%%%%%%%%%%%%%%%%
\section{Proofs} \label{SectionOnProofs}
%%%%%%%%%%%%%%%%%%%%%%%%%%%%%%%%%%%%%%%%%%%%%%%%%%%%%%%%%%%%%%%%%%%%%%%%%%%%%%%%%%%%%%%%%%%%%%%%%%%%%%%%%%%%%%%%%%%%%%%%%%%%%%
This section is devoted to the proofs of Theorems \ref{main},  \ref{Estimate} and \ref{PertainingToCablingConjecture}.
\subsection{Proof of Theorem \ref{main}} Let $p,q$ be relatively prime, non-zero integers with $q>0$, and write $r=p/q$. We shall work through the various cases in Theorem \ref{main} separately. The computation of the rank of $\widehat{HF}(S^3_{p/q}(K),[i])$ appeals to Theorems \ref{RationalSurgeryTheorem} and \ref{RankOfMappingCone} for help, which in tandem assert  
$$\text{rk }\widehat{HF}(S^3_{p/q}(K),[i]) = \text{rk }[ \kerr (\hat{\mathbb D}_{r,[i]})_*] + \text{rk } [\cokerr (\hat{\mathbb D}_{r,[i]})_*].$$
It is the ranks of the kernel and cokernel of $(\hat{\mathbb D}_{p/q,[i]})_*$ that we shall compute explicitly.  To cut down on notation we introduce the shorthand symbols 
$$ \hat H^A_s: = H_*(\hat A_s ) \quad \text{ and } \quad  \hat H^B_s: = H_*(\hat B_s)\cong \mathbb Z.$$
\vskip1mm
\noindent {\bf Case 1: $\mathbf{\nu>0}$ and $\mathbf {0<(2\nu -1)q\le p}$. } Fix a spin$^c$-structure $[i]\in\mathbb Z/p \mathbb Z$, let 
$$s_0=\max\left\{ s\in \mathbb Z\, \big|\, \left\lfloor \frac{i+ps}{q}\right\rfloor < \nu\right\},$$
and note that the assumption $(2\nu-1)q\le p$ implies that one of the two mutually exclusive cases musts occur:
\vskip2mm
\indent {\bf\bf Subase ($\alpha$)}  $-\nu< \left\lfloor \frac{i+ps_0}{q} \right\rfloor$ and $\left\lfloor \frac{i+p(s_0-1)}{q} \right\rfloor \le -\nu$. \\
%%%
%%%
\indent{\bf Subase ($\beta$)} $\left\lfloor \frac{i+ps_0}{q} \right\rfloor \le -\nu$.  
\vskip2mm
To see why these are the only two possibilities, assume that Subase ($\beta$) didn't occur. Then $-\nu < \lfloor \frac{i+ps_0}{q}\rfloor$ and we must show that $\lfloor \frac{i+p(s_0-1)}{q}\rfloor \le -\nu$, as stipulated in Subcase ($\alpha$). This is an easy computation, relying on the inequality $(2\nu-1)\le \frac{p}{q}$:
$$
\left\lfloor \frac{i+p(s_0-1)}{q}\right\rfloor  = \left\lfloor \frac{i+ps_0}{q} - \frac{p}{q}\right\rfloor \le   \left\lfloor \frac{i+ps_0}{q} \right\rfloor  -  \left\lfloor  \frac{p}{q}\right\rfloor < \nu -(2\nu-1) = -\nu +1.
$$
With this understood, we turn to Subase ($\alpha$). Here $\hat v_{\left\lfloor \frac{i+ps}{q}\right\rfloor, *} =0$ for all $s\le s_0$ (see Remark \ref{RemarkAboutNuAndVsAndHs}), and similarly $\hat h_{\left\lfloor \frac{i+ps}{q}\right\rfloor , *} =0$ for $s\ge s_0$. This is indicated in Figure \ref{pic1a} where dashed arrows represent the zero map (and where vertical arrows correspond to the various $\hat v_{t,*}$ maps, and slanted arrows to the $\hat h_{t,*}$ maps, conventions we use in all diagrams in this proof). Thus the cokernel of $(\hat{\mathbb D}_{p/q,[i]})_*$ is trivial, while its kernel has rank 
\begin{equation} \label{KernelRankInCase1aAlpha}
\text{rk } [\kerr (\hat{\mathbb D}_{r,[i]})_*] =  \text{rk } H^A_{\left\lfloor \frac{i+ps_0}{q}\right\rfloor}  + 
\sum _{s\in \mathbb Z, \, s\ne s_0} \left(\text{rk } H^A_{\left\lfloor \frac{i+ps}{q}\right\rfloor}  -1\right).
\end{equation}
Theorem \ref{RankOfMappingCone} them implies that 
\begin{align} \label{RankComputationForCase1a}
\text{rk } \widehat{HF}(S^3_{p/q}(K),[i]) & = \text{rk } H^A_{\left\lfloor \frac{i+ps_0}{q}\right\rfloor}  + 
\sum _{s\in \mathbb Z, \, s\ne s_0} \left(\text{rk } H^A_{\left\lfloor \frac{i+ps}{q}\right\rfloor}  -1\right),\cr
& = 1+ \sum _{s\in \mathbb Z} \left(\text{rk } H^A_{\left\lfloor \frac{i+ps}{q}\right\rfloor} -1\right),\cr
& = 1+\sum _{t\in \mathbb Z} \varphi_i(t)\left(\text{rk } H^A_t -1\right), \cr
& = 1+\mathcal S_i.
\end{align}
\begin{figure}[htb!] 
\centering
\includegraphics[width=15cm]{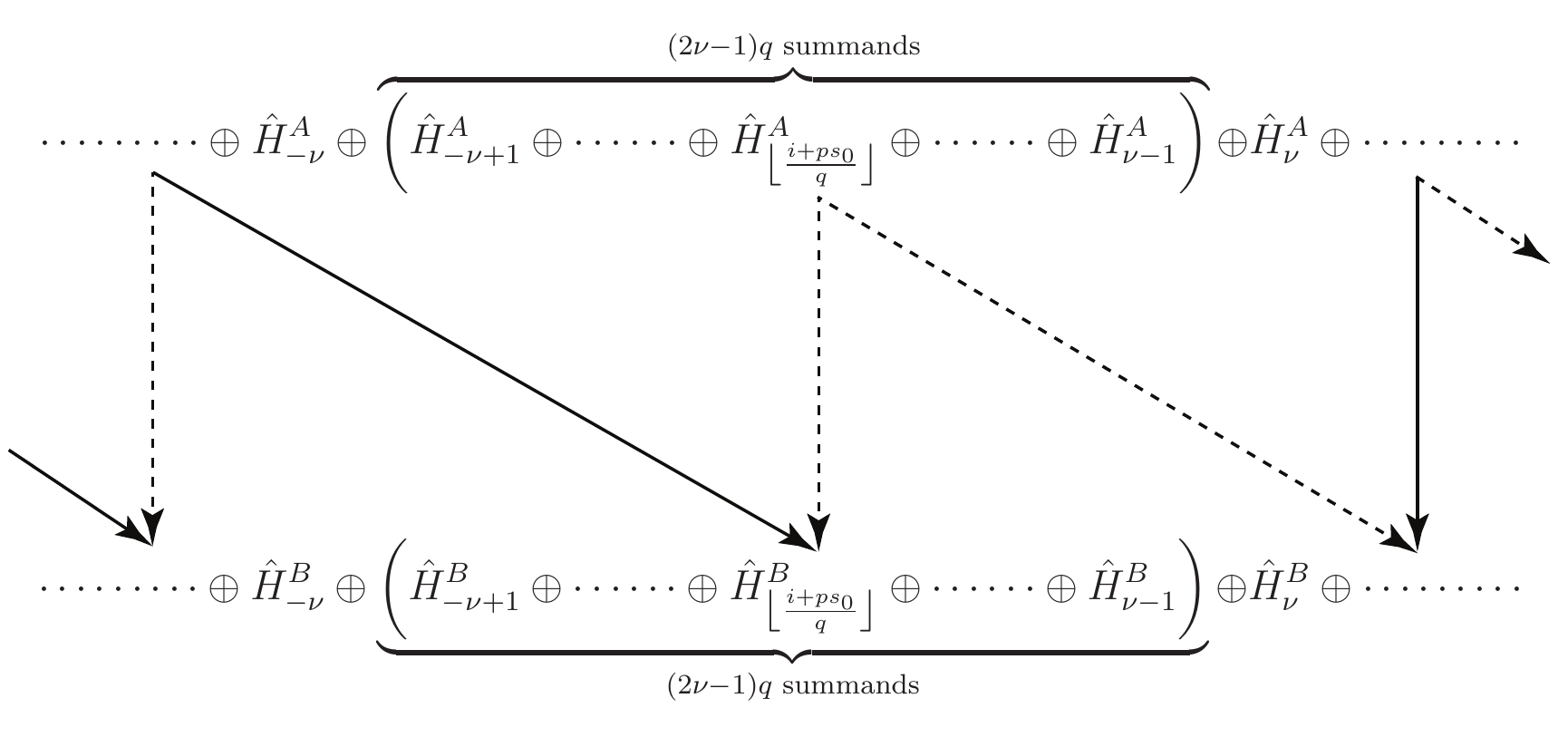}
%\put(-180,130){\tiny Case $(i)$}
%\put(-205,90){\tiny Case $(ii)$}
\caption{{\bf Case 1, Subcase ($\alpha$)}. There is no interaction between the nontrivial vertical and horizontal  homomorphisms $\hat v_{s,*}$ and $\hat h_{s,*}$. }  \label{pic1a}
\end{figure}
%%%
%%%

In Subcase ($\beta$) we find that $\hat v_{\left\lfloor \frac{i+ps}{q}\right\rfloor, *} =0$ for $s\le s_0$ and $\hat h_{\left\lfloor \frac{i+ps}{q}\right\rfloor , *} =0$ for $s > s_0$. Thus, the only interacting vertical and horizontal map (sharing the codomain $\hat H^B_{\lfloor \frac{i+p(s_0+1)}{q} \rfloor }$) are  $\hat h_{\left\lfloor \frac{i+ps_0}{q}\right\rfloor, *}$ and $\hat v_{\left\lfloor \frac{i+p(s_0+1)}{q}\right\rfloor, *}$. This situation is illustrated in Figure \ref{pic1b}, which carries the same notational conventions as Figure \ref{pic1a}. In particular, we find again that the cokernel of $(\hat{\mathbb D}_{r,[i]})_*$ is trivial, while its kernel has the same rank already computed in \eqref{KernelRankInCase1aAlpha}. Given this, the rank of $\widehat{HF}(S^3_r(K),[i])$ for Subcase $(\beta)$ is the same as that compute in Subcase $(\alpha)$ in \eqref{RankComputationForCase1a}.  
\begin{figure}[htb!] 
\centering
\includegraphics[width=16cm]{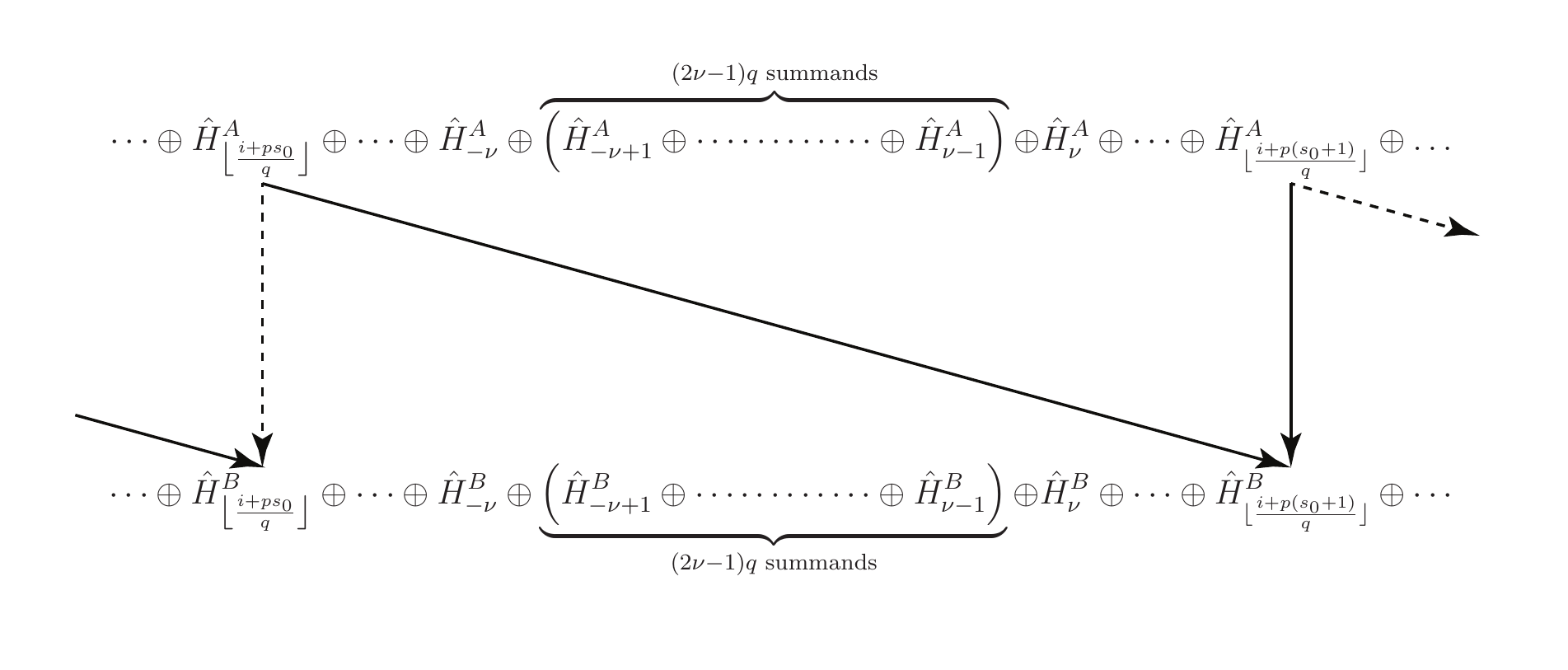}
%%%
%%%
\caption{{\bf Case 1, Subcase ($\beta$)}. There is only one pair of an interacting vertical and horizontal homomorphisms, namely $\hat v_{\left\lfloor \frac{i+p(s_0+1)}{q}\right\rfloor, *}$ and $\hat h_{\left\lfloor \frac{i+ps_0}{q}\right\rfloor, *}$, sharing codomain $\hat H^B_{\lfloor \frac{i+p(s_0+1)}{q} \rfloor }$.}  \label{pic1b}
\end{figure}
%%%
%%%
%%%%%%%%%%%%%%%%%%%%%%%%%%%%%%%%%%%%%%%%%%%%%%%%%%%%%%%%%%%%%%%%%%%%%%%%%%%%%%%%%%%%%%%%%%%%%%%%%%%%%%%%%%%%%%%%%%%%%%%%%%%%%%%%%%%%%%%%%%%%%%%%%%%%%%%%%%%%%%%%%%%%%%%%%%%%%%%%%%%%%%%%%%%%%%%%%%%%%%%%%%%%%%%%%%%%
\vskip1mm
\noindent {\bf Case 2: $\mathbf{\nu>0}$ and $\mathbf {0<p\le(2\nu -1)q}$. } Fix again a spin$^c$-structure $[i]\in\mathbb Z/p \mathbb Z$ and define the integers $s_0$ and $s_1$ as
$$s_0=\max \left\{s\in \mathbb Z \,\big|\, \textstyle \left\lfloor \frac{i+ps}{q}\right\rfloor \le -\nu \right\} 
\quad \text{ and } \quad 
s_1=\min \left\{s\in \mathbb Z \,\big|\, \textstyle\left\lfloor \frac{i+ps}{q}\right\rfloor \ge \nu \right\}. 
$$ 
Since $p, \nu>0$, clearly $s_0<s_1$, and $\lfloor\frac{i+ps}{q} \rfloor > -\nu$ for all $s>s_0$, and similarly $\lfloor\frac{i+ps}{q} \rfloor < \nu$ for $s<s_1$. Thus $\hat v_{\lfloor \frac{i+ps}{q}\rfloor,*}=0$ for $s<s_1$, and $\hat h_{\lfloor \frac{i+ps}{q}\rfloor,*}=0$ for $s>s_0$ (Remark \ref{RemarkAboutNuAndVsAndHs}). It is not hard to see that the condition $0<p\le (2\nu-1)q$ implies that $s_1>s_0+1$. Note that 
\begin{equation} \label{S1MinusS0PlusOne}
s_1-s_0-1 = \# \left\{s\in \mathbb Z\, \big|\, -\nu < \left\lfloor \frac{i+ps}{q} \right\rfloor <\nu \right\} = \sum _{|t|<\nu} \varphi_{[i]}(t).
\end{equation}
Using these, and relying on Figure \ref{pic2} for a visual reference (keeping our conventions from previous figures), it is now easy to compute the ranks of the kernel and cokernel of $(\hat{\mathbb D}_{r,[i]})_*$:
%%%
%%%
\begin{align*}
\text{rk } [\cokerr (\hat{\mathbb D}_{r,[i]})_*] & = s_1-s_0-2,\cr
\text{rk } [\kerr (\hat{\mathbb D}_{r,[i]})_*] & = \sum_{s_0<s<s_1}\left(  \text{rk } H^A_{\left\lfloor \frac{i+ps}{q}\right\rfloor} \right)+  \sum_{s\le s_0 \text{ or } s\ge s_1}\left(  \text{rk } H^A_{\left\lfloor \frac{i+ps}{q}\right\rfloor} -1\right).
\end{align*}
%%%
%%%
Using Theorem \ref{RankOfMappingCone} and equation \eqref{S1MinusS0PlusOne} we now compute the rank of $\widehat{HF}(S_{p/q}^3(K),[i])$:
\begin{align*}
\text{rk } \widehat{HF}& (S^3_{p/q}(K),[i]) = \cr
&=(s_1-s_0-2) +\sum_{s_0<s<s_1}\left(  \text{rk } H^A_{\left\lfloor \frac{i+ps}{q}\right\rfloor} \right)+
\sum_{s\le s_0 \text{ or } s\ge s_1}\left(  \text{rk } H^A_{\left\lfloor \frac{i+ps}{q}\right\rfloor} -1\right) ,\cr
& =2(s_1-s_0-1)-1 +\sum_{s_0<s<s_1}\left(  \text{rk } H^A_{\left\lfloor \frac{i+ps}{q}\right\rfloor} -1  \right)+
\sum_{s\le s_0 \text{ or } s\ge s_1}\left(  \text{rk } H^A_{\left\lfloor \frac{i+ps}{q}\right\rfloor} -1\right) ,\cr
& = -1+2\sum_{|t|<\nu}\varphi_{[i]}(t) +
\sum_{s\in \mathbb Z}\left(  \text{rk } H^A_{\left\lfloor \frac{i+ps}{q}\right\rfloor} -1\right) ,\cr
& = -1+2\sum _{|t|<\nu} \varphi_i(t) + \sum _{t\in \mathbb Z} \varphi_i(t) \left(\text{rk } H^A_t -1\right), \cr
& = -1+2\sum _{|t|<\nu} \varphi_i(t) + \mathcal S_i.
\end{align*}
\begin{figure}[htb!] 
\centering
\includegraphics[width=16cm]{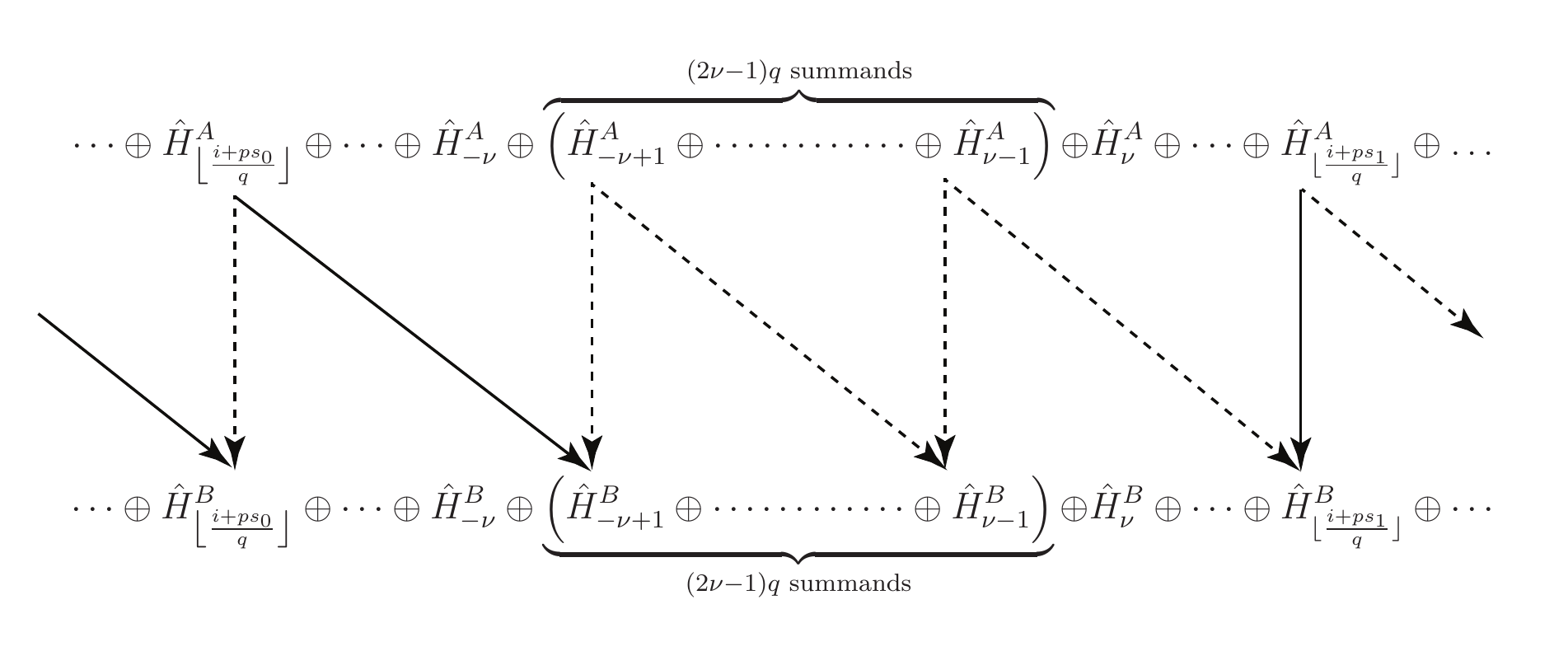}
%%%
%%%
\caption{{\bf Case 2}: There is no interaction between nontrivial vertical and horizontal homomorphisms. Rather, there is an extended region where both types of homomorphisms vanish, contributing to ranks of the kernel and cokernel of $(\hat{\mathbb D}_{r,[i]})_*$.}  \label{pic2}
\end{figure}
%%%%%%%%%%%%%%%%%%%%%%%%%%%%%%%%%%%%%%%%%%%%%%%%%%%%%%%%%%%%%%%%%%%%%%%%%%%%%%%%%%%%%%%%%%%%%%%%%%%%%%%%%%%%%%%%%%%%%%%%%%%%%%%%%%%%%%%%%%%%%%%%%%%%%%%%%%%%%%%%%%%%%%%%%%%%%%%%%%%%%%%%%%%%%%%%%%%%%%%%%%%%%%%%%%%%
\vskip1mm
\noindent {\bf Case 3: $\mathbf{\nu>0}$ and $\mathbf{p<0}$. } Pick a spin$^c$-structure $[i]\in\mathbb Z/p \mathbb Z$ and similarly to the previous case,  define the integers $s_0$ and $s_1$ as
$$s_0=\min \left\{s\in \mathbb Z \,\big|\, \textstyle \left\lfloor \frac{i+ps}{q}\right\rfloor \le -\nu \right\} 
\quad \text{ and } \quad 
s_1=\max \left\{s\in \mathbb Z \,\big|\, \textstyle\left\lfloor \frac{i+ps}{q}\right\rfloor \ge \nu \right\}. 
$$ 
Note that $s_1 < s_0$ and the equality $s_1+1=s_0$ is possible. Since $p$ is now negative we conclude that $\hat v_{\lfloor\frac{i+ps}{q} \rfloor ,* }=0$ for $s>s_1$ while $\hat h_{\lfloor\frac{i+ps}{q} \rfloor , *}=0$ for $s<s_0$. Observe that  
$$ s_0-s_1-1 = \#\left\{s\in \mathbb Z\, \big| \, -\nu < \left\lfloor \frac{i+ps}{q} \right\rfloor <\nu \right\} = \sum_{|t|<\nu}\varphi_{[i]}(t).$$ 
Figure \ref{pic3} gives a visual representation of the situation described above, and aids us in determining the ranks of the kernel and cokernel of $(\hat{\mathbb D}_{r,[i]})_*$:
%%%
%%%
\begin{align*}
\text{rk } [\cokerr (\hat{\mathbb D}_{r,[i]})_*] & = s_0-s_1,\cr
\text{rk } [\kerr (\hat{\mathbb D}_{r,[i]})_*] & = \sum_{s_1<s<s_0}\left(  \text{rk } H^A_{\left\lfloor \frac{i+ps}{q}\right\rfloor} \right)+  \sum_{s\le s_1 \text{ or } s\ge s_0}\left(  \text{rk } H^A_{\left\lfloor \frac{i+ps}{q}\right\rfloor} -1\right).
\end{align*}
%%%
%%%
The rank of $\widehat{HF}(S^3_r(K),[i])$ follows with the help of Theorem \ref{RankOfMappingCone}:
\begin{align*} 
\text{rk } \widehat{HF}(S^3_{p/q}(K),[i]) & =s_0-s_1 +\sum_{s_1<s<s_0}\left(  \text{rk } H^A_{\left\lfloor \frac{i+ps}{q}\right\rfloor} \right)+
\sum_{s\le s_1 \text{ or } s\ge s_0}\left(  \text{rk } H^A_{\left\lfloor \frac{i+ps}{q}\right\rfloor} -1\right) ,\cr
& = 1+2(s_0-s_1-1)+
\sum_{s\in \mathbb Z}\left(  \text{rk } H^A_{\left\lfloor \frac{i+ps}{q}\right\rfloor}  -1\right) ,\cr
& = 1+2\sum _{|t|<\nu} \varphi_i(t) + \sum _{t\in \mathbb Z} \varphi_i(t) \left(\text{rk } H^A_t -1\right), \cr
& = 1+2\sum _{|t|<\nu} \varphi_i(t) + \mathcal S_i.
\end{align*}
%
%%%
%%%
%%%
\begin{figure}[htb!] 
\centering
\includegraphics[width=16cm]{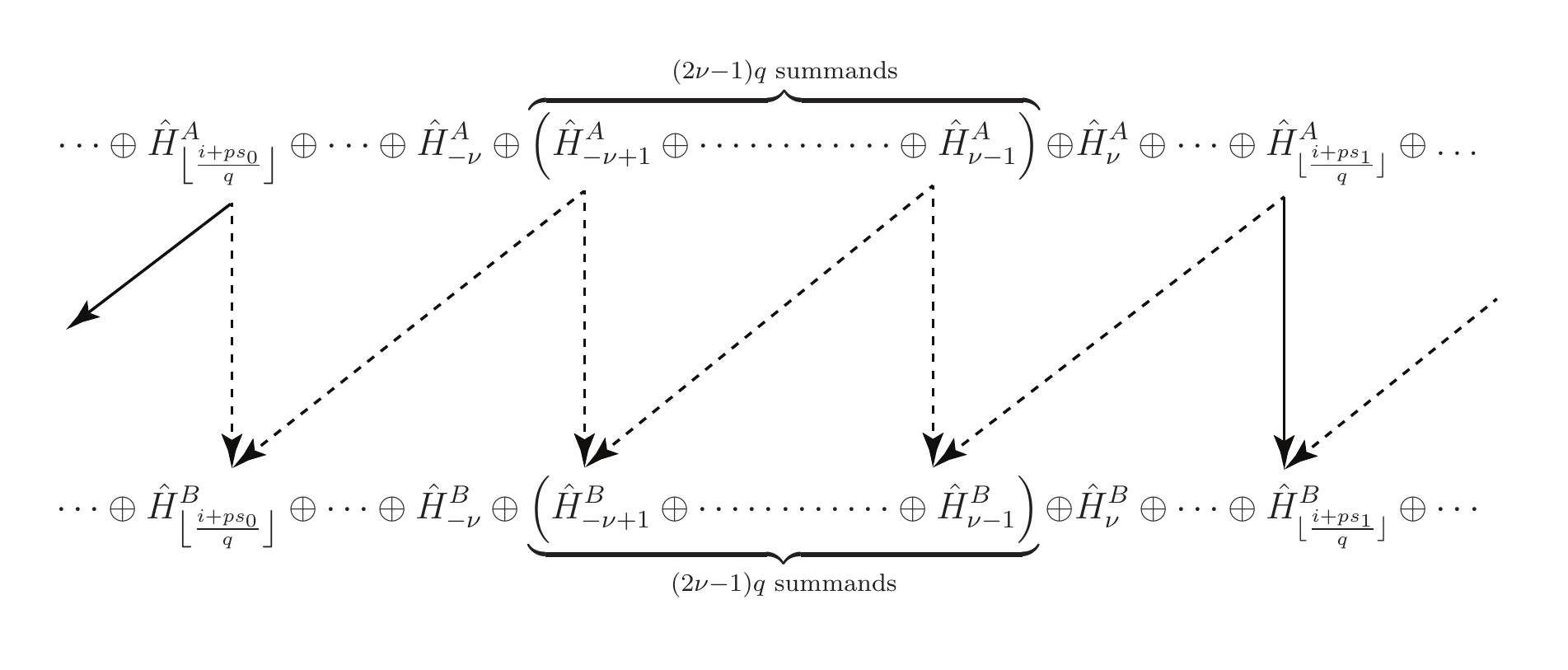}
%%%
%%%
\caption{{\bf Case 3}: Here to, as in Case 2, there is no interaction between nontrivial vertical and horizontal homomorphisms. The same type of region as in Case 2, with vanishing vertical and horizontal homomorphisms, occurs here too, and contributes to the ranks of the kernel and cokernel of $(\hat{\mathbb D}_{r,[i]})_*$..}  \label{pic3}
\end{figure}
%%%%%%%%%%%%%%%%%%%%%%%%%%%%%%%%%%%%%%%%%%%%%%%%%%%%%%%%%%%%%%%%%%%%%%%%%%%%%%%%%%%%%%%%%%%%%%%%%%%%%%%%%%%%%%%%%%%%%%%%%%%%%%%%%%%%%%%%%%%%%%%%%%%%%%%%%%%%%%%%%%%%%%%%%%%%%%%%%%%%%%%%%%%%%%%%%%%%%%%%%%%%%%%%%%%%
\vskip1mm
\noindent {\bf Case 4: $\mathbf{\nu=0}$. } Let $[i]\in \mathbb Z/p\mathbb Z$ be a spin$^c$-structure and define the integers $s_0, s_1 $ as 
$$s_0 =\left\{ 
\begin{array}{cl}
\max \left\{ s \in \mathbb Z \, \big| \, \textstyle \left\lfloor \frac{i+ps}{q}\right\rfloor <0    \right\} &  ; p>0,\cr  & \cr
\min \left\{ s \in \mathbb Z \, \big| \, \textstyle \left\lfloor \frac{i+ps}{q}\right\rfloor <0    \right\} & ;p<0,
\end{array}
\right.
%%%
%%%
\quad  
s_1= \left\{
\begin{array}{cl}
\min \left\{ s \in \mathbb Z \, \big| \, \textstyle \left\lfloor \frac{i+ps}{q}\right\rfloor > 0    \right\} &;p>0, \cr & \cr
 \max \left\{ s \in \mathbb Z \, \big| \, \textstyle \left\lfloor \frac{i+ps}{q}\right\rfloor > 0    \right\} & ; p<0.
\end{array}
\right.
$$
{\bf Subcase $(\gamma)$.} Assume firstly that $p>0$. Then $s_0<s_1$, and $s_1=s_0+1$ is a possibility. Note that $\hat v_{s,*} = 0 $ for $s\le s_0$ and $\hat h_{s,*}=$ for $s_1\le s$. The map 
%
%\begin{equation} \label{RelevantMapForCase2WhenPIsPositive}
$$
\left( \hat h_{\lfloor\frac{i+ps_0}{q} \rfloor ,*} \oplus \dots \oplus \hat v_{\lfloor\frac{i+ps_1}{q} \rfloor ,*} \right):  \bigoplus _{s_0\le s\le s_1} H^A_{\lfloor\frac{i+ps}{q} \rfloor } \longrightarrow \bigoplus_{s_0 < s\le s_1} H^B_{\lfloor\frac{i+ps}{q} \rfloor }
$$
%\end{equation}
%
is onto, regardless of whether $\text{rk }(\hat v_{0,*} + \hat h_{0,*})$ equals $1$ or $2$ (Figure \ref{pic4}), leading to 
%%%
%%%
\begin{align*}
\text{rk } [\cokerr (\hat{\mathbb D}_{r,[i]})_*] & = 0,\cr
\text{rk } [\kerr (\hat{\mathbb D}_{r,[i]})_*] & = \sum_{s_0 \le s \le s_1} \text{rk }H^A_{\left\lfloor\frac{i+ps}{q}\right\rfloor}  - (s_1-s_0) + \sum_{s<s_0 \text{ or } s>s_1} \left( \text{rk }H^A_{\left\lfloor\frac{i+ps}{q}\right\rfloor} -1\right),
\end{align*}
%%%
%%%
from which %these a computation of the rank of $\widehat{HF}(S^3_r(K),[i])$ easily follows:
\begin{align} \nonumber
\text{rk } \widehat{HF}(S^3_{p/q}(K),[i]) & = \sum_{s_0 \le s \le s_1} \text{rk }H^A_{\left\lfloor\frac{i+ps}{q}\right\rfloor}  - (s_1-s_0)+ \sum_{s<s_0 \text{ or } s>s_1} \left( \text{rk }H^A_{\left\lfloor\frac{i+ps}{q}\right\rfloor} -1\right), \cr
 & =1+ \sum_{s\in \mathbb Z} \left( \text{rk }H^A_{\left\lfloor\frac{i+ps}{q}\right\rfloor} -1\right) , \cr
& = 1 +\sum _{t\in \mathbb Z} \varphi _{[i]}(t) \left( H_*(\hat A_t) -1\right), \cr
& = 1+\mathcal S_{[i]},
\end{align}  
follows.
%%%
%%%
\begin{figure}[htb!] 
\centering
\includegraphics[width=14cm]{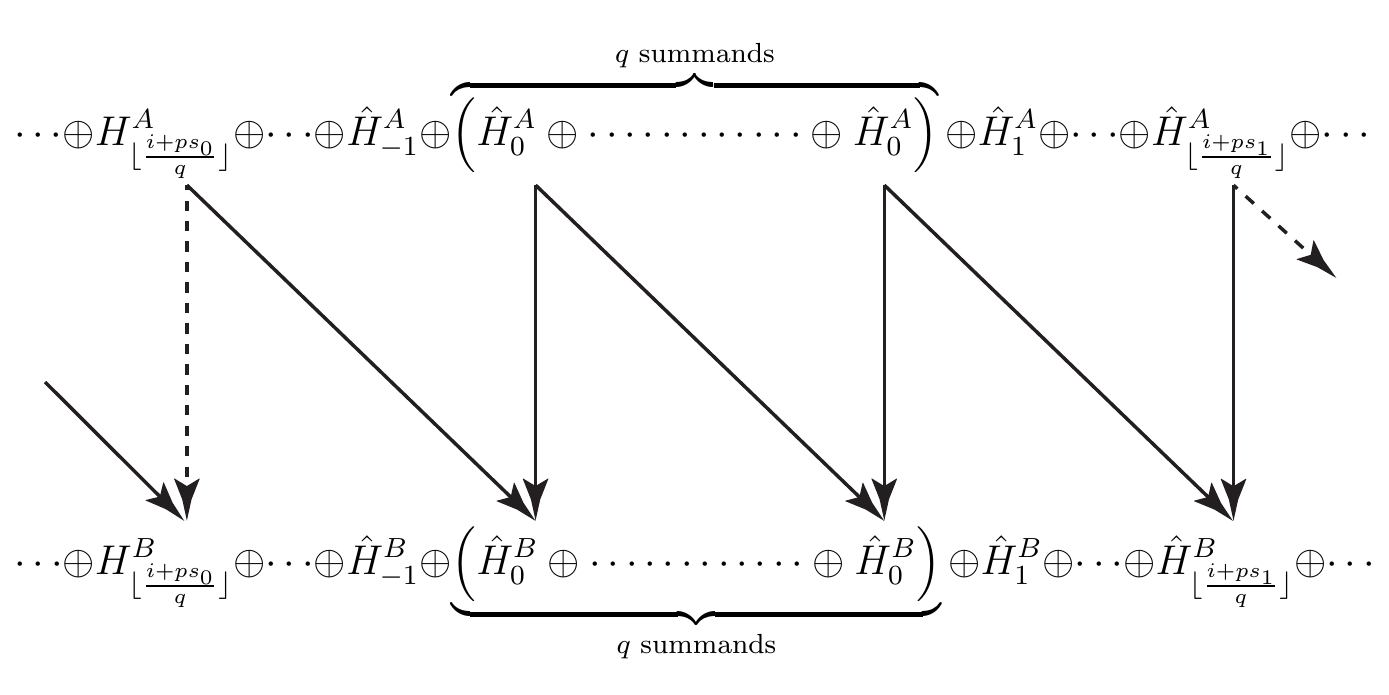}
%%%
%%%
\caption{{\bf Case 4, Subcase $(\gamma)$}: Here  $p>0$ and the rank of $(\hat v_{0,*}+\hat h_{0,*})$ may be 1 or 2. Regardless, the zig-zag region of interacting nontrivial vertical and horizontal homomorphisms is surjective. The example above has $s_1=s_0+3$. }  \label{pic4}
\end{figure}
%%%
%%%
\vskip3mm
Before proceeding to the next subcase, observe that in the event that $p<0$, the condition $0\in \{\lfloor \frac{i+ps}{q}\rfloor \,|\,s\in \mathbb Z\}$ is equivalent to the equality $s_0>s_1+1$. The inequality $s_0\ge s_1+1$ holds always, with equality occurring if and only if $0\notin \{\lfloor \frac{i+ps}{q}\rfloor \,|\,s\in \mathbb Z\}$. Also note that $\hat v_{s,*}=0$ for $s_0\le s$ and $\hat h_{s,*} =0$ for $s\le s_1$.
\vskip1mm
\noindent{\bf Subcase $(\delta)$.} Assume that $p<0$ and that $\text{rk }(\hat v_{0,*} + \hat h_{0,*}) =1$ (see Figure \ref{pic5b} if $s_0>s_1+1$, or Figure \ref{pic5c} if $s_0=s_1+1$), or that $\text{rk }(\hat v_{0,*} + \hat h_{0,*}) =2$ with $s_0=s_1+1$ (Figure \ref{pic5c}). In either case, the map 
\begin{equation} \label{pNegativeMap}
\left( \hat h_{\lfloor\frac{i+p(s_1+1)}{q} \rfloor ,*} \oplus \dots \oplus \hat v_{\lfloor\frac{i+p(s_0-1)}{q} \rfloor ,*} \right):  \bigoplus _{s_1 < s < s_0} H^A_{\lfloor\frac{i+ps}{q} \rfloor } \longrightarrow \bigoplus_{s_1 < s \le s_0} H^B_{\lfloor\frac{i+ps}{q} \rfloor }.
\end{equation}
has cokernel of rank 1, leading to 
%%%
%%%
\begin{align*}
\text{rk } & [\cokerr (\hat{\mathbb D}_{r,[i]})_*]  = 1,\cr
\text{rk } & [\kerr (\hat{\mathbb D}_{r,[i]})_*]  = \left(  \sum_{s_1<  s < s_0} \text{rk }H^A_{\left\lfloor\frac{i+ps}{q}\right\rfloor}  - (s_1-s_0-1)\right) + \sum_{s\le s_1 \text{ or } s\ge s_0} \left( \text{rk }H^A_{\left\lfloor\frac{i+ps}{q}\right\rfloor} -1\right).
\end{align*}
%%%
%%%
Adding these two ranks gives $\text{rk }\widehat{HF}(S^3_r(K),[i]) = 1+\mathcal S_{[i]}$. 
%%%
%%%
\begin{figure}[htb!] 
\centering
\includegraphics[width=14cm]{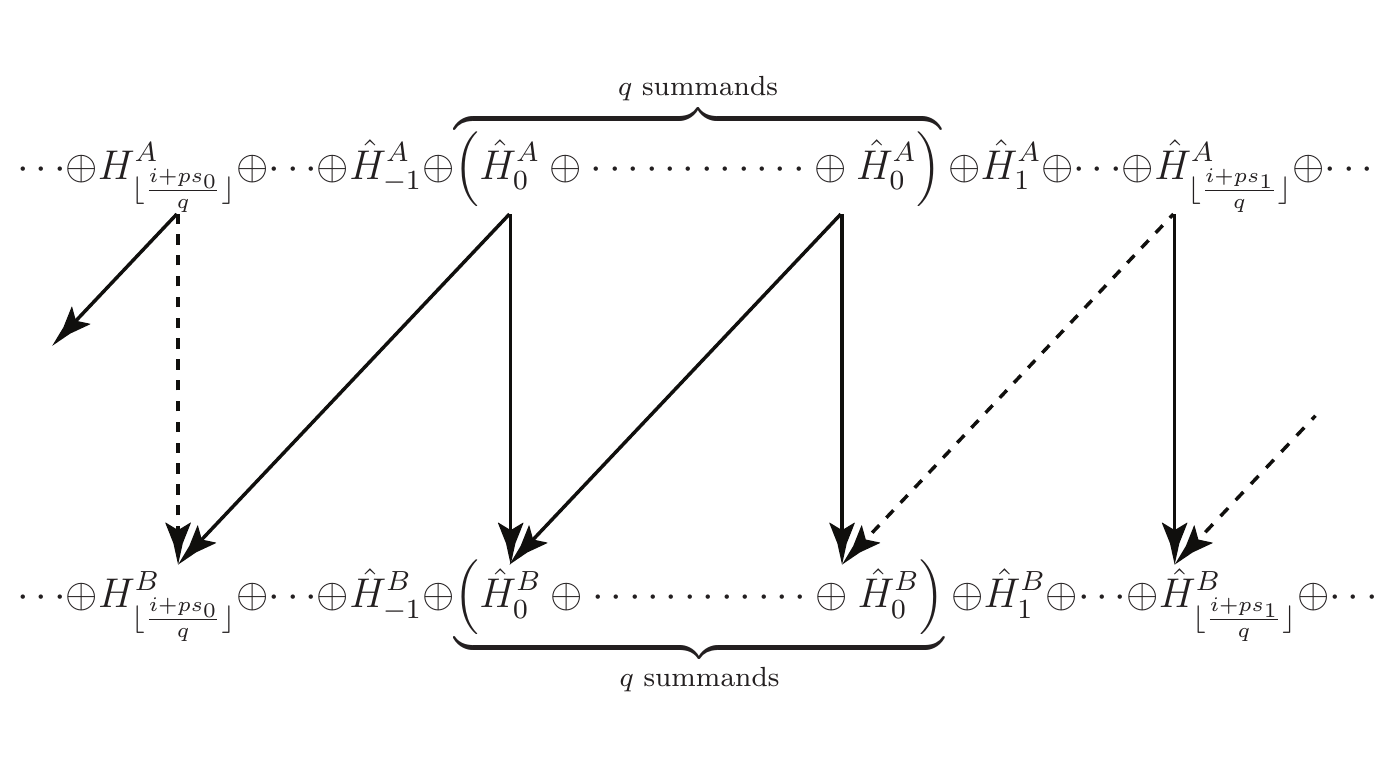}
%%%
%%%
\caption{{\bf Case 4, Subcase $(\delta)$}: This diagram illustrates the case of $p<0$ and $s_0>s_1+1$. The diagram can be used for both cases of $\text{rk }(\hat v_{0,*} + \hat h_{0,*})$ being 1 or 2. }  \label{pic5b}
\end{figure}

%%%
%%%
\begin{figure}[htb!] 
\centering
\includegraphics[width=14cm]{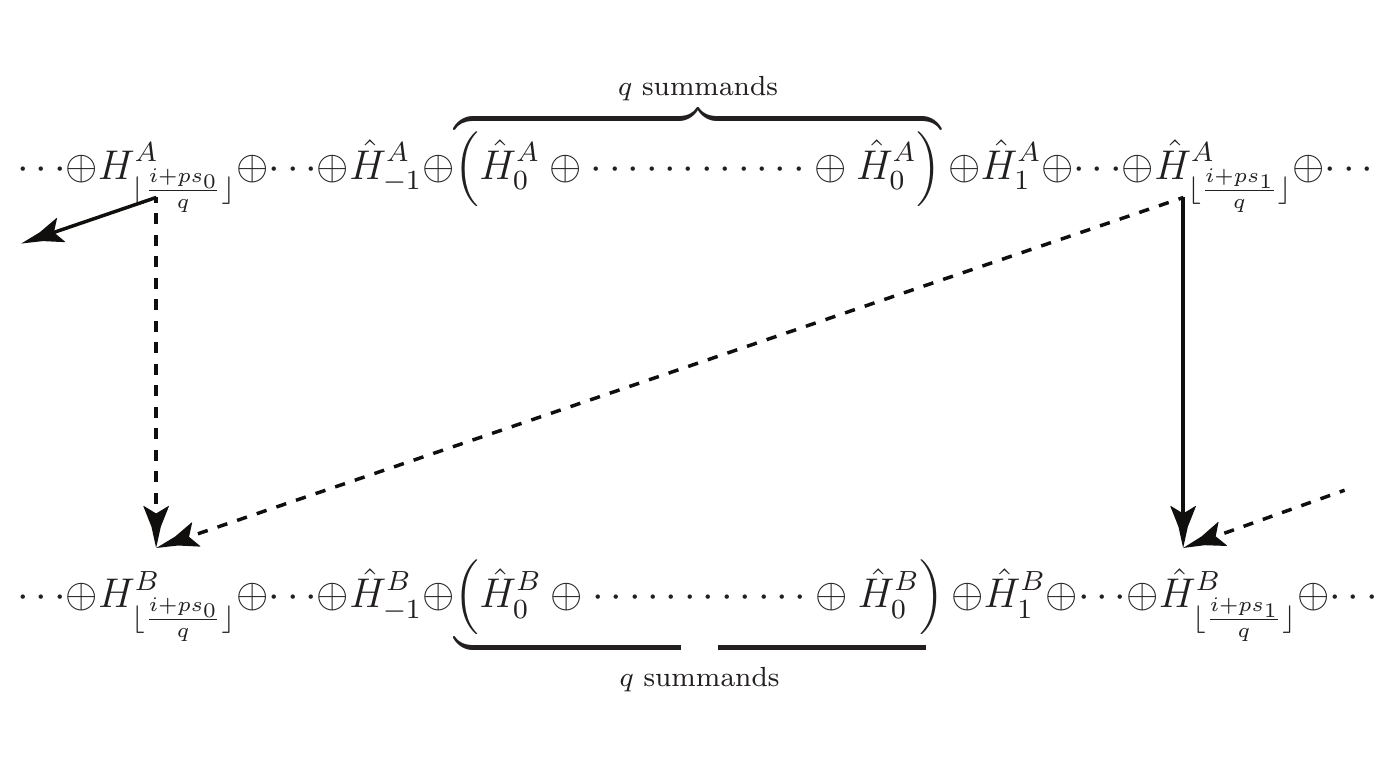}
%%%
%%%
\caption{{\bf Case 4:} This diagram illustrates the case of $p<0$ and $s_0=s_1+1$. The diagram can be used for both cases of $\text{rk }(\hat v_{0,*} + \hat h_{0,*})$ being 1 or 2. }  \label{pic5c}
\end{figure}
%%%
%%%

{\bf Subcase $(\epsilon)$.} Assume that $p<0$ and that $\text{rk }(\hat v_{0,*}+\hat h_{0,*}) =2$ with $s_0>s_1+1$. Then the homomorphism \eqref{pNegativeMap} (compare Figure \ref{pic5b}) is surjective, leading to 
%%%
%%%
\begin{align*}
\text{rk } [\cokerr (\hat{\mathbb D}_{r,[i]})_*] & = 0,\cr
\text{rk } [\kerr (\hat{\mathbb D}_{r,[i]})_*] & = \left( \sum_{s_1<  s < s_0} \text{rk }H^A_{\left\lfloor\frac{i+ps}{q}\right\rfloor}  - (s_1-s_0) \right) + \sum_{s\le s_1 \text{ or } s\ge s_0} \left( \text{rk }H^A_{\left\lfloor\frac{i+ps}{q}\right\rfloor} -1\right),
\end{align*}
%%%
%%%
from which 
\begin{align} \nonumber
\text{rk } \widehat{HF}(S^3_{p/q}(K),[i]) & =\sum_{s_1<s<s_0} \text{rk }H^A_{\left\lfloor\frac{i+ps}{q}\right\rfloor}  - (s_0-s_1)+ \sum_{s\ge s_0 \text{ or } s\le s_1} \left( \text{rk }H^A_{\left\lfloor\frac{i+ps}{q}\right\rfloor} -1\right), \cr
 & =-1+ \sum_{s\in \mathbb Z} \left( \text{rk }H^A_{\left\lfloor\frac{i+ps}{q}\right\rfloor} -1\right) , \cr
& = -1 +\sum _{t\in \mathbb Z} \varphi _{[i]}(t) \left( H_*(\hat A_t) -1\right), \cr
& = -1+\mathcal S_{[i]}. 
\end{align}  
follows. This completes the proof of Theorem \ref{main}. \hfill \qed
 %%%%%%%%%%%%%%%%%%%%%%%%%%%%%%%%%%%%%%%%%%%%%%%%%%%%%%%%%%%%%%%%%%%%%%%%%%%%%%%%%%%%%%%%%%%%%%%%%%%%%%%%%%%%%%%%%%%%%%%%%%%%%%%%%%%%%%%%%%%%%%%%%%%%%%%%%%%%%%%%%%%%%%%%%%%%%%%%%%%%%%%%%%%%%%%%%%%%%%%%%%%%%%%%%%%%%%%%%%%%%%%%%%%%%%%%%%%%%%%%%%%%%%%%%%%%%%%%%%%%%%%%%
\subsection{Proof of Theorem \ref{Estimate}}
%%%
%%%
Without loss of generality we may assume that $|p|>1$ as in the case of $p=\pm 1$, Theorem \ref{Estimate} is trivially true. Recall that $\mathcal R$ is the set of different ranks of the groups $\widehat{HF}(S^3_r(K),[i])$, $[i]\in \mathbb Z/p\mathbb Z$, and where $r=p/q$ (with $p,q$ relatively prime nonzero integers with $q$ positive), and that Theorem \ref{Estimate} claims that $g+1$ is an upper bound for the cardinality of $\mathcal R$. Remark \ref{RemarkAboutBoundOfSpinCStructures} points out that $|\mathcal R|$
is bounded from above by $\frac{|p|+1}{2}$ if $p$ is odd, and is bounded by $\frac{|p|+2}{2}$ is $p$ is even. Therefore it suffices to prove Theorem \ref{Estimate} under the added restriction of $|p|\ge 2g+1$, which we assume throughout this section. 

The upper bound $g+1$ of $|\mathcal R|$ shall be obtained by studying the jumps in ranks as the spin$^c$-structure $[i]$ is changed to $[i+1]$ for $i=1,\dots,|p|-1$. Toward that goal, we introduce the integers $\rho_{[i]}$ as
\begin{equation} \label{DefinitionOfRho}
\rho_{[i]} = \text{rk }\widehat{HF}(S^3_r(K),[i]) - \text{rk }\widehat{HF}(S^3_r(K),[i-1]), 
\end{equation} 
for $[i]\in \mathbb Z/p\mathbb Z - \{0\}$. 

Going forward, it shall prove advantageous to represent spin$^c$-structures from $\mathbb Z/p\mathbb Z$ by their unique representative $i\in \{0,\dots,|p|-1\}$. To underscore this, we shall write $\mathcal S_i$, $\varphi_i(s)$, and $\rho_i$ instead of $\mathcal S_{[i]}, \varphi_{[i]}(s)$ and $\rho_i$ respectively. We shall also assume, without loss of generality, that $K$ is a knot with $\nu(K) \ge \nu(-K)$, so as to exclude the case of $\nu=0$, $p<0$ and $0\in \{ \lfloor \frac{i+ps}{q}\rfloor \, |\,s\in \mathbb Z\}$ from Theorem \ref{main}. We adopt these conventions for the remainder of the present section.  

The proof of Theorem \ref{Estimate} is broken up into a series of lemmas. 

%%%
%%%
\begin{lemma} \label{LemmaAux1}
For $q,\mu>0$ and nonzero $p$, write 
$$(2\mu-1)q = np +r, \quad \text{ and } \quad (\mu-1)q = kp+t,  \quad \text{ with } n,k\in \mathbb Z\text{ and } r,t \in \{0,...,|p|-1\}.$$
If $r>0$ then
$$\sum_{|s|<\mu} \varphi_{[i]}(s) = \left\{
\begin{array}{ll}
\left\lfloor \frac{(2\mu-1)q}{|p|} \right\rfloor + 1 &\quad ; \quad  i\equiv p-t+j \, (\text{mod } p ), \mbox{ for } j=0,\dots,r-1, \cr
& \cr
\left\lfloor \frac{(2\mu-1)q}{|p|} \right\rfloor  &\quad ; \quad  \text{otherwise}. 
\end{array}
\right.
$$
If $r=0$ then $\sum_{|s|<\mu} \varphi_{[i]}(s) = \left\lfloor \frac{(2\mu-1)q}{|p|} \right\rfloor $ for all $[i]\in \mathbb Z/p\mathbb Z$.
\end{lemma}
%%%
%%%
\begin{proof}
By definition of $\varphi _i(s)$, we find that 
$$\sum _{|s|<\mu}\varphi _i(s) = \text{Cardinality of the set}  \left\{ s\in \mathbb Z \, \big|\, -\mu+1\le  \left\lfloor \frac{i+ps}{q} \right\rfloor \le \mu -1\right\}.  $$
The double inequality on the right-hand side above prompts us to count integers $s$ with $-q(\mu-1) \le i+ps < \mu q$. We view this as the problem of counting the number of \lq\lq steps\rq\rq that land in the set $\mathcal I = \{-(\mu-1)q,...,\mu q -1\}$ when \lq\lq walking\rq\rq through the integers by taking steps of length $p$ and making sure that one step lands on $i$ (we don't need $i\in \mathcal I$). Putting the problem this way, and since $|\mathcal I|=(2\mu-1)q$, it should be clear that the answer is either $\lfloor \frac{(2\mu-1)q}{|p|}\rfloor$ or $\lfloor \frac{(2\mu-1)q}{|p|}\rfloor + 1$, with the latter occurring for those values of $i$ that are congruent mod $p$ to the first $r$ values (or last $r$ values) of $\mathcal I$. If $r=0$, no such $i$ exists, thereby proving the lemma.
\end{proof}
%%%
%%%
\begin{lemma} \label{LemmaAux2}
Let $p,q$ be two relatively prime, nonzero integers with $q>0$ and $|p|>1$.
\begin{itemize}
\item[(i)] For $i\in \{1,\dots,|p|-1\}$, we obtain 
$$
\varphi_{i}(s)  - \varphi _{i-1}(s) = \left\{
\begin{array}{rl}
1 & \quad ; \quad i \equiv sq \,\, (\text{mod } p), \cr
-1 & \quad ; \quad i \equiv (s+1)q \, \,(\text{mod } p), \cr
0 & \quad ;\quad \text{otherwise.}
\end{array}
\right.
$$
The two congruences $i \equiv sq \, (\text{mod } p)$ and $i \equiv (s+1)q \, (\text{mod } p)$ are mutually exclusive.  
%%%
%%%
\item[(ii)] Let $g$ be a positive integer such that $|p|\ge 2g+1$,  and let $\mu$ be an integer with $0<\mu \le g$. Then for any $i\in \{1,\dots,|p|-1\}$ we obtain
$$
\sum _{|s|<\mu} \varphi_i(s) - \sum _{|s|<\mu} \varphi_{i-1}(s) = \left\{
\begin{array}{rl}
1 & \quad ;\quad i \equiv (-\mu+1)q \, \,(\text{mod }p), \cr
-1 & \quad ; \quad i \equiv \mu q\,\, (\text{mod }p), \cr
0 & \quad ; \quad \text{otherwise}.
\end{array}
\right.
$$

\end{itemize}
\end{lemma}
%%%
%%%
\begin{proof}
{\em Case (i).}  Assume firstly that $p>0$. Then the value of $\varphi_i(s)$ is obtained by counting integers $t$ in the range
\begin{equation}\label{FirstCountOfIntegers}
\frac{sq-i}{p} \le t < \frac{(s+1)q-i}{p}.
\end{equation}
Similarly, the value of $\varphi _{i-1}$ is the count of integers $t$ in the range
\begin{equation} \label{SecondCountOfIntegers}
\frac{sq-i+1}{p} \le t < \frac{(s+1)q-i+1}{p}.
\end{equation}
The difference $\varphi_i(s) - \varphi _{i-1}(s)$ can only be nonzero if there are integers $t$ in the range \eqref{FirstCountOfIntegers} that do not occur in the range \eqref{SecondCountOfIntegers}, or vice versa. 

Specifically, any integer $t$ appearing in  \eqref{FirstCountOfIntegers} but not in \eqref{SecondCountOfIntegers}, must satisfy the double inequality $\frac{sq-i}{p} \le t < \frac{sq-i+1}{p}$. Multiplying his by $p$ yields the double inequality of integers $sq-i \le pt < sq-i+1$, with the left-most and right-most entry differing by 1. The only possible solution is $sq-i = pt$, or equivalently, $i\equiv sq \, (\text{mod } p)$. 

Similarly, any integer $t$ appearing in \eqref{SecondCountOfIntegers} but not in \eqref{FirstCountOfIntegers}, must satisfy the double inequality $\frac{(s+1)q-i}{p} \le t < \frac{(s+1)q-i+1}{p}$. Multiplying his by $p$ yields again a double inequality of integers $(s+1)q-i \le pt < (s+1)q-i+1$, with the left-most and right-most entry differing by 1. This again leads to an equality, namely $(s+1)q-i = pt$, or equivalently, $i\equiv (s+1)q \, (\text{mod } p)$. 

It is not possible for there to be an integer $t$ appearing in \eqref{FirstCountOfIntegers} but not in \eqref{SecondCountOfIntegers}, and another integer $t$ appearing in \eqref{SecondCountOfIntegers} but not in \eqref{FirstCountOfIntegers}, for such integers would force the congruence $(s+1)q\equiv sq \, (\text{mod }p)$, leading to $p|q$, and hence to $|p|=1$, contrary to assumption. 

Lastly, the existence of an integer $t$ in the range \eqref{FirstCountOfIntegers} that isn't in the range  \eqref{SecondCountOfIntegers} implies $\varphi_i(s) - \varphi_{i+1}(s) = 1$. Similarly, the existence of an integer $t$ in the range \eqref{SecondCountOfIntegers} not occurring in the range  \eqref{FirstCountOfIntegers} yields $\varphi_i(s) - \varphi_{i+1}(s) = -1$, proving Case (i) of the lemma for $p>0$. The case of $p<0$ is handled verbatim.
\vskip2mm
%%%
%%%
{\em Case (ii).}  Assume firstly again that $p>0$. Then the sum $\sum _{|s|<\mu} \varphi_i(s)$ is the count of integers $t$ in the range 
$$ \frac{(-\mu +1)q-i}{p} \le t<\frac{\mu q -i}{p},$$
while the sum $\sum _{|s|<\mu} \varphi_{i-1}(s)$ is the count of integers $t$ in the range
$$  \frac{(-\mu +1)q-i+1}{p} \le t<\frac{\mu q -i+1}{p}.$$
The difference $\sum _{|s|<\mu} \varphi_i(s) - \sum _{|s|<\mu} \varphi_{i-1}(s)$ is therefore the count of integers $t$ in the range $\frac{(-\mu +1)q-i}{p} \le t <  \frac{(-\mu +1)q-i+1}{p}$, minus the count of integers $t'$ in the range $\frac{\mu q -i}{p} \le t' < \frac{\mu q -i+1}{p}$. Multiplying the first of these double inequalities by $p$, yields the double inequality of integers $(-\mu+1)q-i \le pt < (-\mu+1)q-i+1$, with the left-most and right-most sides differing by 1. This forces $pt = (-mu+1)q-i$, showing that there is a unique integer $t$ precisely when $i\equiv (-\mu+1)q\,\, (\text{mod }p)$.   The other double inequality similarly leads to a unique integer solution $t'$ with $pt' = \mu q-i$, occurring precisely when $i \equiv \mu q \,\,(\text{mod }p)$. Both of these congruences cannot occur simultaneously, for if they did they would imply $(2\mu -1)q\equiv 0\, (\text{mod }p)$, forcing $2\mu-1\equiv 0 \,(\text{mod }p)$.  Since $0<\mu\le g$ and $p\ge 2g+1$, this latter congruence is not possible. The case of $p<0$ follows analogously. 
\end{proof}
Going forward we shall rely on the integers $\bar q, \bar p$, with $\bar q \in \{1,\dots, |p|-1\}$, determined by the equation 
\begin{equation} \label{DefinitionOfQbarAndPbar}
q\bar q + p\bar p =1. 
\end{equation}
%%%
%%%
\begin{lemma} \label{LemmaAux3}
Let $p,q$ be relatively prime, nonzero integers with $q>0$, and let $\bar q, \bar p$ be as in \eqref{DefinitionOfQbarAndPbar}. 
\begin{itemize}
\item[(i)] If $\nu = 0$ or if $\nu>0$ and $(2\nu-1)q\le p$, then for any $i\in \{1,\dots,|p|-1\}$, we obtain
\begin{equation} \label{FormulaForRhoCasei}
\rho_i= \sum _{k\in \mathbb Z} \left( \text{rk }\hat H^A_{i\bar q +pk} \, -  \text{rk } \hat H^A_{i\bar q -1+pk}\right).
\end{equation}
%
%%%
%%%
\item[(ii)] Assume that $|p|\ge 2g+1$, that $\nu >0$ and either $p<0$ or $0<p\le (2\nu -1)q$. Then  
\begin{equation} \label{FormulaForRhoCaseii}
\rho _i = \left\{
\begin{array}{rl}
1 + \sum _{k\in \mathbb Z} \left( \text{rk }\hat H^A_{i\bar q +pk} \, -  \text{rk } \hat H^A_{i\bar q -1+pk}\right) & \quad ; \quad  i\equiv (-\nu+1)q \, (\text{mod } p), \cr 
-1 + \sum _{k\in \mathbb Z} \left( \text{rk }\hat H^A_{i\bar q +pk} \, -  \text{rk } \hat H^A_{i\bar q -1+pk}\right) & \quad ; \quad   i\equiv \nu q \, (\text{mod } p), \cr 
\sum _{k\in \mathbb Z} \left( \text{rk }\hat H^A_{i\bar q +pk} \, -  \text{rk } \hat H^A_{i\bar q -1+pk}\right) & \quad ; \quad   \text{otherwise}.
\end{array}
\right.
\end{equation}
\end{itemize}
In both cases, if $i,j \not \equiv q \, (\text{mod }p)$, then  $\rho_j = - \rho_{i}$ for $j \equiv q- i \, (\text{mod } p)$. 
\end{lemma}
%%%
%%%
The sums on the right-hand sides in \eqref{FormulaForRhoCasei} and \eqref{FormulaForRhoCaseii} are finite, as $\text{rk}\,\hat H^A_t = 1$ whenever $|t|\ge g(K)$. 
%%%
%%%
\begin{proof} 
{\em Case (i).} In case (i), according to Theorem \ref{main}, $\rho_i$ equals the difference $\mathcal S_i-\mathcal S_{i-1}$, and can thus be computed as 
$$\rho_i = \sum _{s\in \mathbb Z} \left( \varphi_{i}(s) - \varphi_{i-1}(s)\right) [\text{rk }\hat H^A_{s}-1] $$
Equation \eqref{FormulaForRhoCasei} now follows directly from Case (i) in Lemma \ref{LemmaAux2}. Next, pick integers $i,j$ with $j\equiv q-i \,\,(\text{mod } p)$. Then 
\begin{align*}
\rho_j & = \sum _{k\in \mathbb Z} \left( \text{rk } \hat H^A_{j\bar q +pk} - \text{rk } \hat H^A_{j\bar q -1 +pk}  \right), \cr  
& = \sum _{k\in \mathbb Z} \left( \text{rk } \hat H^A_{(q-i)\bar q +pk} - \text{rk } \hat H^A_{(q-i)\bar q -1 +pk}  \right), \cr  
& = \sum _{k\in \mathbb Z} \left( \text{rk } \hat H^A_{-i\bar q +1 +pk} - \text{rk } \hat H^A_{-i\bar q  +pk}  \right), \cr  
& = \sum _{k\in \mathbb Z} \left( \text{rk } \hat H^A_{i\bar q -1 +pk} - \text{rk } \hat H^A_{i\bar q  +pk}  \right), \cr  
& = -\rho_i. 
\end{align*}
%%%
%%%
\vskip2mm
{\em Case (ii).} In Case (ii) of the present lemma, Theorem \ref{main} implies that $\rho_i$ is given by 
\begin{equation} \label{AuxiliaryEquation1}
\rho _i = \sum _{|s|<\nu} \left(\varphi_i(s) - \varphi_{i-1}(s)  \right) +  \sum _{s\in \mathbb Z} \left( \varphi_{i}(s) - \varphi_{i-1}(s)\right) [\text{rk }\hat H^A_{s}-1].
\end{equation}
The second summand equals $\rho_i$ from Case (i), and is thus given by the right-hand side of \eqref{FormulaForRhoCasei}. Using Case (ii) in Lemma \ref{LemmaAux2} to rewrite the first summand on the right-hand side of \eqref{AuxiliaryEquation1}, equation \eqref{FormulaForRhoCaseii} follows. 

The equality $\rho_j = -\rho_i$ for $j\equiv q-i \, (\text{mod }p)$ follows exactly as in Case (i), with the observation that if $i\equiv \nu q\, (\text{mod }p)$, then $q-i \equiv q-\nu q \, (\text{mod }p) \equiv (-\mu +1)q \, (\text{mod }p)$.  
\end{proof}
%%%
%%%
\begin{definition}
Let $p,q$ be relatively prime integers with $q>0$. Define $\tilde q\in \{1,\dots, |p|-1\}$ to be the unique integer with $\tilde q \equiv q \,(\text{mod }p)$. For $i\in \{1,\dots, |p|-1\}-\{\tilde q\}$, define its {\em dual} $\tilde i\in \{1,\dots, |p|-1\}-\{\tilde q\}$ by the requirement $\tilde i \equiv q-i \, (\text{mod }p)$.  
\end{definition}
%%%
%%%
The relevance of the dual indices comes of course from Lemma \ref{LemmaAux2} by which $\rho\,_{\tilde i} = -\rho_i$, for $i\in \{1,\dots, |p|-1\}-\{\tilde q\}$. Remember that at the beginning of this section we assumed $|p|\ge 2g+1$. This condition guarantees that there is at most one nonzero term in the sums on the right-hand sides of \eqref{FormulaForRhoCasei} and \eqref{FormulaForRhoCaseii}, indeed many $\rho_i$ vanish. 
%%%
%%%
\begin{lemma} \label{LemmaAux4}
The set $\{\rho _i\, |\, i=1,\dots, |p|-1\}$ contains at most $2g-1$ nonzero terms. Of these, $2g-2$ are made up of $g-1$ pairs of opposite integers.   
\end{lemma}
%%%
%%%
\begin{proof}
We note that the map $i\mapsto iq \,(\text{mod } p)$ is a bijection of $\{1,\dots, |p|-1\}$, and so it suffices to prove the lemma by for the set $\{\rho_{iq} \, |\, i=1, \dots ,|p|-1\}$ instead of $\{\rho _i\, |\, i=1,\dots, |p|-1\}$ (where by $\rho_{iq}$ we mean $\rho_j$ for the unique $j\in \{1,\dots, |p|-1\}$ with $iq \equiv j \,(\text{mod }p)$). 

If either $\nu=0$ or $\nu>0$ and $(2\nu-1)q\le p$, then formula \eqref{FormulaForRhoCasei} for $\rho_{iq}$ becomes
$$\rho_{iq} = \sum _{k\in \mathbb Z} \left( \text{rk } \hat H^A_{i+pk} -  \text{rk } \hat H^A_{i-1+pk}  \right),$$ 
showing that $\rho_{iq} =0$ for $i=g+1, \dots, |p|-g$, while $\rho_{iq} = -\rho_{jq}$ with $j\equiv 1-i \,(\text{mod }p)$ and for $i=2,\dots, g$. When $p$ is odd there is one \lq\lq self-dual\rq\rq value of $i$, namely $i_0=\frac{|p|+1}{2}$. However, since $|p|\ge 2g+1$, the value $i_0$ lies in the set $\{g+1,\dots,|p|-g\}$. %We note the $\rho_q$ is left without a \lq\lq dual pair\rq\rq. 

If $\nu>0$ and  $p\le (2\nu-1)q$ then formula \eqref{FormulaForRhoCasei} becomes 
$$
\rho _{iq} = \left\{
\begin{array}{rl}
1 + \sum _{k\in \mathbb Z} \left( \text{rk }\hat H^A_{i +pk} \, -  \text{rk } \hat H^A_{i -1+pk}\right) & \quad ; \quad  i\equiv -\nu+1 \, (\text{mod } p), \cr 
-1 + \sum _{k\in \mathbb Z} \left( \text{rk }\hat H^A_{i +pk} \, -  \text{rk } \hat H^A_{i -1+pk}\right) & \quad ; \quad   i\equiv \nu  \, (\text{mod } p), \cr 
\sum _{k\in \mathbb Z} \left( \text{rk }\hat H^A_{i\bar q +pk} \, -  \text{rk } \hat H^A_{i\bar q -1+pk}\right) & \quad ; \quad   \text{otherwise},
\end{array}
\right.
$$
showing again that $\rho_i=0$ for $i=g+1,\dots, |p|-g$, as in the previous case. The equality $\rho_{iq} = -\rho_{jq}$ continues to hold when $j\equiv 1-i \,(\text{mod }p)$ and for $i=2,\dots, g$ in the present case as well. The only observation we need to make here is that the indices $i=|p|-\nu+1$ and $j=\nu$ satisfy the congruence from the previous sentence. 
\end{proof}                                                                                                                 
%%%
%%%
The content of the next lemma is the observation that dual pairs of indices $i, \tilde i$ occur in a \lq\lq nested\rq\rq manner. 
%%%
%%%
\begin{lemma} \label{LemmaAux5}
Let $i,j\in \{1,\dots,|p|-1\}-\{\tilde q\}$ be such that $i<j<\tilde i$. Then $i<\tilde j <\tilde i$. 
\end{lemma}
%%%
%%%
\begin{proof}
Consider first the case of $1\le i<\tilde q$. Then $\tilde i = q-i$ and if $j$ lies between $i$ and $\tilde i$, the so does $\tilde j$ as $\tilde j = q-j$. If $i>\tilde q$ then $\tilde i = q-i+kp$ for some integer $k$, independent of $i$. If $j$ lies between $i$ and $\tilde i$, and since $\tilde j = q-j+kp$, then $\tilde j$ also lies between $i$ and $\tilde i$. 
\end{proof}
%%%
%%%
\vskip2mm
We are now in a position to conclude the proof of Theorem \ref{Estimate}. For convenience set $\rho_0 =\text{rk }\widehat{HF}(S^3_r(K),[0])$, and recall that $\rho _i$ for $i\in \{1,\dots, |p|-1\}$ is defined by \eqref{DefinitionOfRho}. Thus 
$$\text{rk } \widehat{HF}(S^3_r(K),[i]) = \rho_0+ \rho_1+\dots + \rho_i. $$
As there are at most $2g-1$ nonzero values among $\rho_1,\dots,\rho_{|p|-1}$, and of these $2g-2$ appear as pairs of opposite integers (Lemma \ref{LemmaAux4}), with nested distribution (Lemma \ref{LemmaAux5}), the number of different values of the sums $\sum _{j=0}^i \rho_j$ for $i=0,\dots ,|p|-1$ it at most $g+1$, as claimed in Theorem \ref{Estimate}. \hfill \qed
%%%%%%%%%%%%%%%%%%%%%%%%%%%%%%%%%%%%%%%%%%%%%%%%%%%%%%%%%%%%%%%%%%%%%%%%%%%%%%%%%%%%%%%%%%%%%%%%%%%%%%%%%%%%%%%%%%%%%%%%%%%%%%%%%%%%%%%
\subsection{Proof of Theorem \ref{PertainingToCablingConjecture}}
%%%
%%%
The paragraphs of Section \ref{SectionOnCablingConjecture} preceding Theorem \ref{PertainingToCablingConjecture} explain that any knot $K$ for which the result of integral Dehn surgery $S^3_p(K)$ is reducible, leads to the decomposition $S^3_p(K) \cong Y\#L(a,b)$ for some relatively prime, non-zero integers $a, b$, with $a >1$. Clearly $a$ must divide $p$, and we note that $Y$ is a rational homology sphere with $H_1(Y;\mathbb Z) \cong \mathbb Z_{|p|/a}$. Any spin$^c$-structure $\s$ on $S^3_p(K)$ may be written as $s=\s_1\# \s_2$ with $\s_1\in Spin^c(Y)$ and $\s_2\in Spin^c(L(a,b))$. Theorem 1.5 in \cite{Peter2} implies 
$$\text{rk } \widehat{HF}(S^3_p(K),\s) = \left( \text{rk } \widehat{HF}(Y,\s_1)\right) \cdot \left( \text{rk } \widehat{HF}(L(a,b),\s_2)\right),$$
leading to the equality 
$$\text{rk } \widehat{HF}(S^3_p(K),\s_1\#\s_2) =  \text{rk } \widehat{HF}(Y,\s_1), \quad \text{ for all } \s_2 \in Spin^c(L(a,b)).$$
As $|Spin^c(L(a,b))|=a$, it follows that each of the $|p|$ terms $\text{rk } \widehat{HF}(S^3_p(K),\s)$ occurs with a multiplicity divisible by $a$. Adding or subtracting a constant from all the terms $\text{rk } \widehat{HF}(S^3_p(K),\s)$ does not affect this multiplicity property. 

As in Corollary \ref{CorollaryAboutIntegerSurgeries}, write $2\nu-1=np+r$ with $n\in \mathbb Z$ and $r\in \{0\dots, |p|-1\}$, and let $\mathcal I_{[i]}$ be the set $\mathcal I_{[i]} = \{ s\in \mathbb Z \, |\, |s|<g \text{ and } s\equiv i \, (\text{mod } p )\}$. Furthermore, let $\mathcal J\subset \mathbb Z/p\mathbb Z$ be the set of indices $\mathcal J = \{ [-\nu+j]\, |\, j=1,\dots, r\}$. If $r=0$ then $\mathcal J=\emptyset$. 

If $\nu>0$ and $2\nu -1\le p$ then Corollary \ref{CorollaryAboutIntegerSurgeries} gives 
\begin{equation} \label{AuxEquation1}
\sum _{s\in \mathcal J_{[i]}} (\text{rk } \hat H^A_s - 1) =  \text{rk }\widehat{HF}(S^3_p(K),[i]) - 1. 
\end{equation}
If $\nu = 0$ Corollary \ref{CorollaryAboutIntegerSurgeries} implies that 
\begin{equation} \label{AuxEquation2}
\sum _{s\in \mathcal J_{[i]}} (\text{rk } \hat H^A_s - 1) = \left\{
\begin{array}{cl}
 \text{rk }\widehat{HF}(S^3_p(K),[i]) + 1 & ; \quad \begin{array}[t]{l} p< 0, \text{rk }(\hat v_0+\hat h_0)=2, \text{ and }\cr
\lfloor \frac{i+ps}{q}\rfloor = 0 \text{ for some } s\in \mathbb Z, \end{array} \cr
 \text{rk }\widehat{HF}(S^3_p(K),[i]) - 1 & ; \quad \text{otherwise.}
\end{array}
\right.
\end{equation}
In the remaining case of $\nu>0$ and $p\le 2\nu-1$, Corollary \ref{CorollaryAboutIntegerSurgeries} yields 
\begin{equation} \label{AuxEquation3}
 \text{rk }\widehat{HF}(S^3_p(K),[i]) - (2n-1)   = \left\{
\begin{array}{cl}
2+ \sum _{s\in \mathcal J_{[i]}} (\text{rk } \hat H^A_s - 1)  & ; \quad [i]\in \mathcal J,\cr
\sum _{s\in \mathcal J_{[i]}} (\text{rk } \hat H^A_s - 1)  & ; \quad [i]\notin \mathcal J,
\end{array}
\right.
\end{equation}
where $\mathcal J\subset \mathbb Z/p\mathbb Z$ is given as $\mathcal J = \{ [-\nu+1+j]\, |\, j=0,\dots, r-1\}$ when $r>0$, and $\mathcal J=\emptyset$ if $r=0$. Equations \eqref{AuxEquation1}, \eqref{AuxEquation2}, \eqref{AuxEquation3} prove Theorem \ref{PertainingToCablingConjecture}. \hfill \qed

\end{document}